\documentclass[tbtags,reqno]{amsart}
\usepackage{epsfig,amsthm,amsopn,amsfonts,amssymb}

\newtheorem{theorem}{Theorem}
\newtheorem{lemma}[theorem]{Lemma}
\newtheorem{proposition}[theorem]{Proposition}
\newtheorem{example}[theorem]{Example}
\newtheorem{conjecture}[theorem]{Conjecture}

\newtheorem{corollary}[theorem]{Corollary}

\newtheorem{definition}[theorem]{Definition}

\newtheorem{property}[theorem]{Property}

\theoremstyle{remark}
\newtheorem{remark}[theorem]{Remark}
\newtheorem*{acknow}{Acknowledgments}

\def \CC {{\mathcal C}}

\def\kbnd {\mathfrak p}
\def\s {\mathfrak s}
\def\core {\mathfrak c}

\def \core {{\mathfrak c}}

\def\gg {\gamma}
\def\aa {\alpha}

\def\shape{ {\rm {shape}}}
\def\weight{ {\rm {weight}}}

\def \dit#1{\text{\it ``#1"}}

\edef\savecatcodeat{\the\catcode`@}
\catcode`\@=11

\def\tb@ifSpecChars#1#2{#1}
\def\tb@ifNoSpecChars#1#2{#2}

\def\tableau{%
  \bgroup
  \@ifstar{\let\Tif\tb@ifNoSpecChars\tb@tableauB}
          {\let\Tif\tb@ifSpecChars\tb@tableauB}}

\def\tb@tableauB{
  \@ifnextchar[{\tb@tableauC}{\tb@tableauC[]}}

\def\tb@tableauC[#1]{\hbox\bgroup%
    \let\\=\cr
    \def\bl{\global\let\tbcellF\tb@cellNF}%
    \def\tf{\global\let\tbcellF\tb@cellH}
    \dimen2=\ht\strutbox \advance\dimen2 by\dp\strutbox%
    \ifx\baselinestretch\undefined\relax%
    \else%
       \dimen0=100sp \dimen0=\baselinestretch\dimen0%
       \dimen2=100\dimen2 \divide\dimen2 by\dimen0%
    \fi%
    \let\tpos\tb@vcenter
    \tb@initYoung
    \tb@options#1\eoo
    \let\arrow\tb@arrow%
    \dimen0=\Tscale\dimen2%
    \dimen1=\dimen0 \advance\dimen1 by \tb@fframe%
    \lineskip=0pt\baselineskip=0pt
    \def\tb@nothing{}%
    \def\endcellno{$\rss\egroup\bss\egroup}
    \def\endcell{\endcellno\kern-\dimen0}
    \def\begincell{\vbox to\dimen0\bgroup\vss\hbox to\dimen0\bgroup\hss$}%
    \let\overlay\tb@overlay%
    \let\fl\tb@fl%
    \let\lss\hss\let\rss\hss\let\tss\vss\let\bss\vss
    \def\mkcell##1{
        \let\tbcellF\tb@cellD
        \def\tb@cellarg{##1}
        \ifx\tb@cellarg\tb@nothing\let\tb@cellarg\tb@cellE\fi%
        \begincell\tb@cellarg\endcellno
        \tbcellF}
    \let\savecellF\tbcellF
     \Tif{\catcode`,=4\catcode`|=\active}{}\tb@tableauD}%

\let\tb@savetableauD\tableauD
{
    \catcode`|=\active \catcode`*=\active \catcode`~=\active%
    \catcode`@=\active
\gdef\tableauD#1{%
  \Tif{
    \mathcode`|="8000 \mathcode`*="8000%
    \mathcode`~="8000 \mathcode`@="8000%
    \def@{\bullet}%
    \let|\cr
    \let*\tf
    \let~\sk
  }{}%
  \tpos{\tabskip=0pt\halign{&\mkcell{##}\cr#1\crcr}}%
  \global\let\tbcellF\savecellF
  \egroup
  \egroup}
}
\let\tb@tableauD\tableauD
\let\tableauD\tb@savetableauD
\let\tb@savetableauD\undefined

\def\tb@options#1{\ifx#1\eoo\relax\else\tb@option#1\expandafter\tb@options\fi}

\def\tb@option#1{%
  \if#1t\let\tpos\tb@vtop\fi
  \if#1c\let\tpos\tb@vcenter\fi
  \if#1b\let\tpos\vbox\fi
  \if#1F\tb@initFerrers\fi
  \if#1Y\tb@initYoung\fi
  \if#1s\tb@initSmall\fi
  \if#1m\tb@initMedium\fi
  \if#1l\tb@initLarge\fi
  \if#1p\tb@initPartition\fi
  \if#1a\tb@initArrow\fi
}

\def\tb@vcenter#1{\ifmmode\vcenter{#1}\else$\vcenter{#1}$\fi}

\def\tb@vtop#1{\hbox{\raise\ht\strutbox\hbox{\lower\dimen0\vtop{#1}}}}

\def\tb@initPartition{\def\Tscale{.3}}
\def\tb@initSmall{\def\Tscale{1}}
\def\tb@initMedium{\def\Tscale{2}}
\def\tb@initLarge{\def\Tscale{3}}

\def\tb@initArrow{\dimen2=1.25em}

\def\tb@initYoung{%
  \def\tb@cellE{}
  \let\tb@cellD\tb@cellN
  \def\sk{\global\let\tbcellF\tb@cellNF}}
\def\tb@initFerrers{%
  \def\tb@cellE{\bullet}
  \let\tb@cellD\tb@cellNF
  \def\sk{\bullet}}

\tb@initMedium

\def\tb@sframe#1{%
  \vbox to0pt{
    \vss
    \hbox to0pt{%
      \hss
      \vbox to\dimen1{
        \hrule depth #1 height0pt
        \vss
        \hbox to\dimen1{
          \vrule width #1 height\dimen1
          \hss
          \vrule width #1
          }%
        \vss
        \hrule height #1 depth 0in
        }%
      \kern-\tb@hframe
      }%
    \kern-\tb@hframe}}

\def\tb@hframe{.2pt}\def\tb@fframe{.4pt}\def\tb@bframe{2pt}
\def\tb@cellH{\tb@sframe{\tb@bframe}}       
\def\tb@cellNF{}                            
\def\tb@cellN{\tb@sframe{\tb@fframe}}       
\let\tbcellF\tb@cellN                       

\def\tb@rpad{1pt}
\def\tb@lpad{1pt}
\def\tb@tpad{1.8pt}
\def\tb@bpad{1.8pt}

\def\tb@overlay{\endcell\@ifnextchar[{\tb@overlaya}{\begincell}}
\def\tb@overlaya[#1]{\vbox to\dimen0\bgroup%
  \tb@overlayoptions#1\eoo%
  \tss\hbox to\dimen0\bgroup\lss$}
\def\tb@overlayoptions#1{\ifx#1\eoo\relax\else\tb@overlayoption#1\expandafter\tb@overlayoptions\fi}

\def\tb@overlayoption#1{
  \if#1t\def\tss{\vskip\tb@tpad}\let\bss\vss\fi
  \if#1c\let\tss\vss\let\bss\vss\fi
  \if#1b\def\bss{\vskip\tb@bpad}\let\tss\vss\fi
  \if#1l\def\lss{\hskip\tb@lpad}\let\rss\hss\fi
  \if#1m\let\lss\hss\let\rss\hss\fi
  \if#1r\def\rss{\hskip\tb@rpad}\let\lss\hss\fi
}

\def\tb@fl{\endcell\begincell\vrule depth 0pt width \dimen0 height \dimen0 \endcell\begincell}

\@ifundefined{diagram}{}{
\def\tb@arrowpad{.5}

\newoptcommand{\tb@arrow}{\@ne}[2]{%
  \endcell
   \begingroup%
   \let\dg@getnodesize\tb@getnodesize
   \dg@USERSIZE=#1\relax%
   \ifnum\dg@USERSIZE<\@ne \dg@USERSIZE=\@ne \fi%
   \dg@parse{#2}%
   \dg@label{\tb@draw{#1}{#2}}}

\def\tb@getnodesize#1#2#3#4#5{\dimen3=\tb@arrowpad\dimen2 #4=\dimen3 #5=\dimen3\relax}
\def\tb@getnodesize#1#2#3#4#5{\ifnum#2=0\ifnum#3=0\tb@getnodesizetail{#4}{#5}\else\tb@getnodesizehead{#4}{#5}\fi\else\tb@getnodesizehead{#4}{#5}\fi}
\def\tb@getnodesizetail#1#2{\dimen3=.5\dimen2 #1=\dimen3 #2=\dimen3}
\def\tb@getnodesizehead#1#2{\dimen3=.5\dimen2 #1=\dimen3 #2=\dimen3}

\def\tb@draw#1#2#3#4{%
        \dg@X=0\dg@Y=0\dg@XGRID=1\dg@YGRID=1\unitlength=.001\dimen0%
        \dg@LBLOFF=\dgLABELOFFSET \divide\dg@LBLOFF\unitlength%
        \dg@drawcalc
        \begincell
        \let\lams@arrow\tb@lams@arrow
        \begin{picture}(0,0)\begingroup\dg@draw{#1}{#2}{#3}{#4}\end{picture}%
        \endcell
        \endgroup
        \begincell}
}

\def\tb@lams@arrow#1#2{%
 \lams@firstx\z@\lams@firsty\z@
 \lams@lastx#1\relax\lams@lasty#2\relax
 \lams@center\z@
 \N@false\E@false\H@false\V@false
 \ifdim\lams@lastx>\z@\E@true\fi
 \ifdim\lams@lastx=\z@\V@true\fi
 \ifdim\lams@lasty>\z@\N@true\fi
 \ifdim\lams@lasty=\z@\H@true\fi
 \NESW@false
 \ifN@\ifE@\NESW@true\fi\else\ifE@\else\NESW@true\fi\fi
 \ifH@\else\ifV@\else
  \lams@slope
  \ifnum\lams@tani>\lams@tanii
   \lams@ht\ten@\p@\lams@wd\ten@\p@
   \multiply\lams@wd\lams@tanii\divide\lams@wd\lams@tani
  \else
   \lams@wd\ten@\p@\lams@ht\ten@\p@
   \divide\lams@ht\lams@tanii\multiply\lams@ht\lams@tani
  \fi
 \fi\fi
 \ifH@  \lams@harrow
 \else\ifV@ \lams@varrow
 \else \lams@darrow
 \fi\fi
}

\catcode`\@=\savecatcodeat
\let\savecatcodeat\undefined

\begin{document}

\title[
Combinatorics of the $K$-theory of affine Grassmannians]
{Combinatorics of the $K$-theory of affine Grassmannians}

\author{Jennifer Morse}
\thanks{Research supported in part by NSF grant \#DMS-0652668}
\address{Department of Mathematics, Drexel University, 
Philadelphia, PA 19104}
\email{morsej@math.drexel.edu}

\begin{abstract}
We introduce a family of tableaux that simultaneously generalizes
the tableaux used to characterize Grothendieck polynomials
and $k$-Schur functions.  We prove that the polynomials drawn from
these tableaux are the affine Grothendieck polynomials and $k$-$K$-Schur
functions -- Schubert representatives for the $K$-theory of affine 
Grassmannians and their dual in the nil Hecke ring.
We prove a number of combinatorial properties including Pieri rules.
\end{abstract}


\maketitle

\section{Introduction and background}

The Schur functions form a fundamental basis for the symmetric 
function space $\Lambda$.  Many problems in geometry and 
representation theory have been solved using the combinatorics 
behind Schur functions.  Natural combinatorics associated to the 
more general families of Grothendieck polynomials and $k$-Schur functions
has similarly led to an understanding of geometric and representation 
theoretic questions.   Here we explore the underlying combinatorics of 
two families of affine Grothendieck polynomials.

\subsection{Schur functions}

To begin,
the Schur role in geometry is nicely illustrated by the problem of calculating 
intersections of Schubert varieties.  
The cohomology ring of the Grassmannian $Gr_{\ell n}$ 
has a basis of Schubert classes $\sigma_\lambda$,
indexed by partitions $\lambda\in \mathcal P^{\ell n}$ 
that fit inside an $\ell \times (n-\ell)$ rectangle.  The structure 
constants of $H^*(Gr_{\ell n})$ in this basis:
\begin{equation}
\label{struc}
\sigma_\lambda\cup\sigma_\mu = \sum_{\nu\in\mathcal P^{\ell n}}
c_{\lambda\mu}^\nu \sigma_\nu
\,,
\end{equation}
are the number of points in the intersection of
the Schubert varieties $X_\lambda\cap X_\mu\cap X_{\nu^\perp}$.

\smallskip

There is an isomorphism,
$H^*(Gr_{\ell n})\cong \Lambda/\langle e_{n-\ell+1},\ldots,e_n\rangle$,
where the Schur function $s_\lambda$ maps to the Schubert class
$\sigma_\lambda$ when $\lambda\in\mathcal P^{\ell n}$ and to 
zero otherwise.  Thus,
the structure constants $c_{\lambda\mu}^\nu$ of \eqref{struc} 
are none other than {\it Littlewood-Richardson coefficients} in
the expansion
\begin{equation}
\label{litric}
s_\nu\, s_\mu = 
\sum_{\lambda}
c_{\nu\mu}^\lambda s_\lambda
\,,
\end{equation}
and the Schur functions are representatives of the Schubert classes.

\smallskip

The beauty of this identification is that the geometry can be
studied in parallel to Schur theory, where elegant solutions are 
given by way of combinatorics.  For example, the Schur functions 
can be defined explicitly as the weight generating function of 
semi-standard tableaux:
\begin{equation}
\label{schur}
s_{\lambda}= \sum_{T~semi-standard\atop\shape(T)=\lambda} x^{\weight(T)} \, .
\end{equation}
The ring structure can be determined by the {\it Pieri rule},
\begin{equation}
\label{pieri}
s_\ell \, s_\mu= \sum_{\lambda=\mu+horizontal~\ell-strip} s_\lambda
\,,
\end{equation}
which is
a simple matter of adding $\ell$ boxes to the diagram of $\mu$,
and the Littlewood-Richardson coefficients can be characterized
in terms of skew {yamanouchi} tableaux.
Other properties of Schur functions (and thus the Schubert classes)
also amount to simple combinatorial operations such as
\begin{equation} \label{invo}
\omega s_\lambda = s_{\lambda'}\,,
\end{equation}
where $\omega$ is the algebra automorphism defined by 
sending $e_\ell$ to  $h_\ell$.

\subsection{Grothendieck polynomials}
\label{introgrot}

Lascoux and Sch\"utzenberger introduced the Grothendieck 
polynomials in \cite{LSGrot} as representatives for the 
$K$-theory classes determined by structure sheaves of Schubert 
varieties.  Grothendieck polynomials are connected to 
combinatorics, representation theory, and algebraic 
geometry in a way that mimics ties between these theories and
Schur functions (e.g. \cite{Dem,KKGrot,LGrot,FK}).  
This study leads to a generalization of Schubert
calculus and combinatorics is again at the forefront.
For example, the stable Grothendieck polynomials $G_\lambda$ are 
inhomogeneous symmetric polynomials 
whose lowest homogeneous component is a Schur function.
They are characterized by Buch as the 
weight generating function:
\begin{equation}
\label{defgroth}
G_\lambda = \sum_{T~set-valued\atop \shape(T)=\lambda} 
(-1)^{|\lambda|-|\weight(T)|}\,x^T \,,
\end{equation}
where {\it set-valued tableaux} contain the
semi-standard tableaux as a subset.
The Pieri rules are in terms of binomial numbers
\cite{[Lenart]}  and there is a natural 
generalization for yamanouchi tableaux 
\cite{Buch} that gives a combinatorial rule for the structure
constants.

\smallskip

Contrary to Schur functions, Grothendieck polynomials are not
self-dual with respect to the Hall-inner product $\langle\, ,\rangle$.  
This gives rise to the family of polynomials $g_\lambda$, dual to $G_\mu$,
whose {\it top} homogeneous component is a Schur function.  Although less 
well-explored, the theory of these {\it dual Grothendieck polynomials}
is equally as interesting.

\subsection{$k$-Schur functions}

There is a generalization of Schur functions along other lines 
that arose circuitously in a study of Macdonald polynomials 
\cite{[LLM]}.  Pursuant work led to a new basis for
$\mathbb Z[h_1,\ldots,h_k]$ that satisfies properties analogous 
to \eqref{struc} -- \eqref{invo} in an affine setting.  In particular, 
the {\it $k$-Schur functions} $s_\lambda^{(k)}$ were introduced 
in \cite{[LMproofs]} and defined by inverting the system:
\begin{equation}
\label{kschurintro}
h_\lambda = \sum_{\mu} K^{(k)}_{\mu\lambda}\,
s_\mu^{(k)}
\,,
\end{equation}
where $K^{(k)}_{\mu\lambda}$ enumerate 
a family of 
tableaux in bijection with elements of 
the type-$A$ affine Weyl group called {\it $k$-tableaux}
\cite{[LMcore]}.

\smallskip

Geometrically, it was proven that the $k$-Schur functions
are fundamental to the structure of the quantum and affine
(co)homology of Grassmannians analogous to the Schur role in the 
usual cohomology.
Quantum cohomology originated in string theory and symplectic geometry
and is connected through the work of Konsevitch and Manin to the
Gromov-Witten invariants.  
It was shown in \cite{[LMhecke]} that certain Gromov-Witten invariants 
and calculation in the quantum cohomology of the Grassmannian can be 
reduced to computing the product of {$k$-Schur functions}.
It was then shown in \cite{Lam} that $k$-Schur functions 
are the Schubert basis for homology of the affine Grassmannian.  
Again, combinatorics behind $k$-Schur functions is key to their 
study, as well as to the  geometry.  Properties such as their
Pieri rule are proven in \cite{[LMproofs]}.

\smallskip

A second affine analog for Schur functions was introduced
in \cite{[LMhecke]}.  These {\it dual $k$-Schur functions} 
(or {\it affine Schur functions}) can be defined by
$\langle s_\lambda^{(k)},\mathfrak S_\mu^{(k)}\rangle = \delta_{\lambda\mu}$.
These also have significance in the geometry of the affine Grassmannian
and are studied by way of combinatorial identities such as 
their weight generating function
\begin{equation}
\mathfrak S_\lambda^{(k)} = \sum_{T~k-tableaux\atop 
\shape(T)=\core(\lambda)} x^{\weight(T)}
\end{equation}
and their Pieri rule (see \cite{[LLMS]}).

\subsection{Affine Grothendieck polynomials}

The extension of ideas in $k$-Schur theory to an
inhomogeneous setting underlies our investigation of affine 
combinatorics in the K-theoretic framework.  We present a family 
of {\it affine set-valued tableaux} (or {\it affine s-v tableaux}) 
that simultaneously generalizes those used to characterize Grothendieck 
polynomials {\it and} $k$-Schur functions.  We produce
a bijection between affine s-v tableaux and certain elements
that arise from the affine nil Hecke algebra.  From this, we prove that
the polynomials drawn from these tableaux:
\begin{equation}
G_\lambda^{(k)} = \sum_{T~affine~\text{s-v}~tableaux\atop
\shape(T)=\core(\lambda)}
\!  \!
\!  \!
(-1)^{|\lambda|+\weight(T)}\,x^{\weight(T)}
\,,
\end{equation}
are affine stable Grothendieck polynomials introduced in \cite{[Lam]}.
We also study their dual with respect to the Hall-inner product,
the $k$-$K$-Schur functions $g_\lambda^{(k)}$.

\smallskip

We prove that the affine s-v tableaux associated to integer $k>0$ 
contain $k$-tableaux as a subset and reduce to set-valued tableaux 
when $k$ is large.  As a consequence,
affine Grothendieck polynomials and $k$-$K$-Schur 
functions reduce to Grothendieck polynomials and their dual in a 
limiting case.  
Moreover, the term of lowest degree in $G_\lambda^{(k)}$ is 
the dual $k$-Schur function $\mathfrak S_\lambda^{(k)}$ and
the highest term of $g_\lambda^{(k)}$ is the $k$-Schur function
$s_\lambda^{(k)}$.

\smallskip

We also give a number of combinatorial properties for the $k$-$K$-Schur 
functions such as Pieri rules.
In particular, for $k$-bounded partition $\lambda$ and $r\leq k$,
\begin{equation}
\label{kpieriintro}
g^{(k)}_r\,g_\lambda^{(k)} = \sum_{(\mu,\rho)\in \mathcal H_{\lambda,r}^{k}}
(-1)^{r+|\lambda|-|\mu|}
\,g_\mu^{(k)}
\; \;
\text{and}
\end{equation} 
\begin{equation}
g_{1^r}^{(k)}\,g_\lambda^{(k)} = 
\sum_{(\mu,\rho)\in \mathcal E_{\lambda,r}^{k}} 
(-1)^{r+|\lambda|-|\mu|}\,
g_{\mu}^{(k)} \,,
\end{equation}
where the elements of $\mathcal H_{\lambda,r}^{k}$ and
$\mathcal E_{\lambda,r}^{k}$ are 
obtained by way of an affine set-valued notion of 
horizontal and vertical strips, respectively.
In addition to extending \eqref{schur} and \eqref{pieri}
to the affine $K$-theoretic setting, we find that $k$-$K$-Schur
functions satisfy a natural analog to \eqref{invo}.  The image of 
$g_\lambda^{(k)}$ under an involution $\Omega$ on $\Lambda$ is 
simply another $k$-$K$-Schur function:
\begin{equation}
\label{invoO}
\Omega g_\lambda^{(k)} = g_{\lambda^{\omega_k}}^{(k)}\,,
\end{equation}
where $\lambda^{\omega_k}$ is a certain unique
``$k$-conjugate" partition associated to $\lambda$.

\smallskip

Our results establish that the $k$-$K$-Schur functions and
affine Grothendieck polynomials are the affine $K$-theoretic 
Schur functions in a combinatorial sense.  Lam, Schilling, 
and Shimozono show in \cite{[LSS]} that these polynomials 
satisfy an analog geometrically along the lines of
\eqref{struc} and \eqref{litric}.
We thus have a combinatorial framework within which to study
this geometry as has been so fruitful in classical Schur
function theory.

\smallskip
\subsection{Related work}

This study grew out of an FRG problem solving session in Vi\~na del 
Mar, Chile (2008) during which Thomas Lam posed the problem of 
exploring polynomials to play the Schur role in an affine 
$K$-theoretic setting.  Our results, establishing that
the affine Grothendieck polynomials and $k$-$K$-Schur
functions are the appropriate candidate, are obtained purely 
combinatorially.  Recent work of Lam, Schilling, and
Shimozono \cite{[LSS]} carries out a similar investigation
from the geometric viewpoint and they prove that
$G_\lambda^{(k)}$ and $g_\lambda^{(k)}$
are Schubert representatives for K-theory classes of the affine Grassmannian
and their dual in the nil Hecke ring, respectively.

\smallskip

Another direction explores ties between
the theories of Grothendieck and Macdonald polynomials.
$k$-Schur functions are conjectured to be the $t=1$ specialization 
of atoms, a family of polynomials that arose in the study
of Macdonald polynomials \cite{[LLM]}.  
In a forthcoming paper with Jason Bandlow \cite{[BM]}, 
we conjecture that the Macdonald polynomials 
can be expanded positively (up to degree-alternating sign) in
terms of the $\{g_\lambda\}_{\lambda}$ basis.
We prove such an expansion in the Hall-Littlewood case using
a statistic on set-valued tableaux that naturally generalizes 
the Lascoux-Sch\"utzenberger charge \cite{LSfoulkes}.  
In addition, we explore a one parameter 
family of affine Grothendieck polynomials by
$t$-generalizing the methods of \cite{[LMproofs]} where $k$-Schur 
functions were defined so they could be connected to geometry.

\smallskip

\begin{acknow}
I would like to thank the NSF/FRG for their support and all
the members of the FRG: Jason Bandlow, Thomas Lam, Luc Lapointe, 
Huilan Li, Anne Schilling, Mark Shimozono, 
Nicolas Thiery, and Mike Zabrocki. The many discussions with 
Patrick Clarke were also extremely helpful.
\end{acknow}

\section{Definitions  }

\subsection{Partitions}
A partition is an integer sequence 
$\lambda=(\lambda_1\geq\dots\geq\lambda_m>0)$
whose degree is $|\lambda|=\lambda_1 +\cdots +\lambda_m$ and whose 
length $\ell(\lambda)$ is $m$.  Each partition $\lambda$ has an 
associated Ferrers shape with $\lambda_i$ lattice squares in the $i^{th}$ 
row, from the bottom to top.  
Given  a partition $\lambda$, its conjugate $\lambda'$ is the shape
obtained by reflecting  $\lambda$ about the diagonal.  A partition 
$\lambda$ is {\it $k$-bounded} if $\lambda_1 \leq k$ and $\mathcal P^k$
denotes the set of all such partitions.  Any lattice 
square in the shape is called a cell, where the cell $(i,j)$ 
is in the $i$th row and $j$th column of the shape. We say 
that $\lambda \subseteq \mu$ when $\lambda_i \leq \mu_i$ for
all $i$.  When $\rho \subseteq \gamma$, the skew shape $\gg/\rho$ 
is the set theoretic difference $\gg-\rho$.
Dominance order $\lambda\unrhd\mu$ 
on partitions is defined by 
$|\lambda|=|\mu|$ and $\lambda_1+\cdots+\lambda_i\geq
\mu_1+\cdots+\mu_i$ for all $i$. 

\smallskip

Given a partition $\gg$, 
a {\it $\gamma$-removable} corner is a cell $(i,j)\in \gg$ 
with $(i ,j+1),(i+1,j)\not\in \gg$ and a {\it $\gamma$-addable} 
corner is a square  $(i,j)\not\in \gg$ with $(i ,j-1),(i-1,j)\in \gg$.  
A cell $(i,j)\in \gg$ where $(i+1,j+1)\not\in \gg$ is called {\it extremal}.  
In particular, removable corners are extremal. 
In the skew shape $(5,3,3,2)/(1,1)$ below,
all addable corners are labeled by $a$, extremals labeled by $e$,
and removable corners are framed. 
\begin{equation}
{\tiny{\tableau[scY]{\bl,\bl {{a}} , \bl  ,\bl | 
\bl,e, \tf e, \bl {{a}} ,\bl |
\bl,,{{e}},\tf e,\bl {{}},\bl  , \bl , \bl |
\bl ,\bl,,{{e}},\bl {{a}},\bl ,\bl |
\bl,\bl,, {{e}}, e, \tf e,\bl {{a}}}}}
\end{equation}

\smallskip
The hook-length of a cell $c$ in a partition $\gg$
is the number of cells above and to the right of $c$,
including $c$ itself. $h(\gg)$ is the hook-length
of $c=(0,0)$.
A {\it $p$-core} is a 
partition that does not contain any cells with hook-length $p$.
Let $\CC^{p}$ denote the collection of $p$-cores.
The {\it $p$-residue} of square  $(i,j)$ is $j-i \mod p$.
That is, the integer in this square when squares are
periodically labeled with $0,1,\ldots,p-1$, and zeros lie on 
the main diagonal.  The 5-residues of the 5-core
$(6,4,3,1,1,1)$ are
$$
{\tiny{\tableau[scY]{\bl 4 |0|1|2,\bl 3|3,4,0,\bl 1|4,0,1,2,
\bl 3| 0,1,2,3,4,0,\bl 1}}}
$$
Hereafter we work with a fixed integer $k>0$ and all cores/residues 
are $k+1$-cores/$k+1$-residues.  For convenience, we
refer to a corner of residue $i$ as an {\it $i$-corner} and 
a cell $c$ of residue $i$ is denoted by $c(i)$.
The set of residues used to label cells of the 
shape $\gamma/\beta$ is denoted by $Res(\gamma/\beta)$.

\smallskip

Several basic properties of cores will be used throughout.
For example, Property~15 of \cite{[LMcore]} is particularly
useful in our study:
\begin{property}
\label{thecoreprop}
Given cell $c(i)$ in row $r$ of a $p$-core $\gamma$,
\begin{itemize}
\item
if $c$ lies at the end of its row then all extremals of $p$-residue $i$
in a row higher than $r$ lie at the end of their row
\item
if $c$ lies at the top of its column then all extremals of $p$-residue $i$
in a row lower than $r$ lie at the top of their column.
\end{itemize}
\end{property}
\noindent
Note then that a core never has both an addable and a removable corner 
of the same residue.  Further, given a core $\gamma$ and any addable 
$i$-corner, the shape obtained by adding {\it all} $i$-corners to
$\gamma$ is also core.  
A bijection $\kbnd$ from $k+1$-cores to  $k$-bounded partitions 
was defined in \cite{[LMcore]} by the map
$$
\kbnd (\gamma) = \lambda\,,
$$
where $\lambda$ is obtained by deleting all hooks larger
than $k$ from $\gamma$ and reading the rows of the resulting
skew shape.  We denote the inverse by $\core=\kbnd^{-1}$.
Note that $|\lambda|$ is the number of $k$-bounded hooks in $\core(\lambda)$.
\begin{remark}
\label{addonecore}
(Proposition 22 in \cite{[LMcore]})
If $\gamma$ is obtained by adding all addable
corners of some residue to a core $\beta$, then
$|\kbnd(\gamma)|=|\kbnd(\beta)|+1$.
\end{remark}

Our study also requires the use of compositions where the
length $\ell(\alpha)$ denotes the number of parts in $\alpha$.
We work often with {\it $k$-bounded compositions}; $\alpha$ 
where $\alpha_i\leq k$.

\subsection{Symmetric functions}
Let $\Lambda$ denote the ring of symmetric functions, generated by the
elementary symmetric functions $e_r=\sum_{i_1<\ldots <i_r}x_{i_1}
\cdots x_{i_r}$, or equivalently by the complete functions
$h_r=\sum_{i_1\leq\ldots\leq i_r}x_{i_1}\cdots x_{i_r}$.
The {\it Hall inner product} on $\Lambda$ is defined by
$$
\langle h_\lambda, m_\mu\rangle = \delta_{\lambda\mu}\,,
$$
where $m_\mu$ is a monomial symmetric function.
Complete details on symmetric functions can be found in 
e.g \cite{Macbook,Stanley,Lasnotes}.

\smallskip

{\it Schur functions} $s_\lambda$ are the orthonormal basis 
with respect to the Hall inner product.  They can be combinatorial 
defined using (semi-standard) {\it tableaux} -- the filling of a 
Ferrers shape with integers that strictly increase in columns 
and are not decreasing
in rows.  The weight of tableau $T$ is the composition 
$w(T)=\alpha$ 
where $\alpha_i$ is the multiplicity of $i$ in $T$.
Schur functions are the weight generating functions: 
\begin{equation}
s_\lambda = \sum_{\shape(T)=\lambda}
x^{w(T)}\,.
\end{equation}
Tableaux can equivalently be viewed as a 
sequence of shapes differing by horizontal strips where a 
{\it horizontal $r$-strip} is a skew shape with $r$ cells 
and whose columns have at most one cell.
To be precise, a semi-standard tableau of weight $\alpha$
and shape $\lambda$ is a sequence
$$
\emptyset\subseteq\lambda^{(1)}\subseteq\lambda^{(2)}\subseteq\cdots
\subseteq\lambda^{(\ell(\alpha)}=\lambda\,,
$$
where $\lambda^{(x)}/\lambda^{(x-1)}$ is an $\alpha_x$-strip,
for $x=1,\ldots,\ell(\alpha)$.

\subsection{Set-valued tableaux}

The weight generating function for {\it Grothendieck polynomials}
is given in \cite{Buch} using {\it set-valued tableaux} --
fillings of a Ferrers shape with sets of integers 
where a set $X$ below (west of) $Y$ satisfies $\max X<(\leq)\min Y$.
The weight is again the composition determined by the multiplicities 
of each letter.  For any partition $\lambda$, the
symmetric Grothendieck polynomial is
\begin{equation}
\label{defgrot}
G_\lambda\, =\, \sum_{set~valued~T\atop \shape(T)=\lambda} 
(-1)^{|w(T)|+|\lambda|}\,x^{T}
\,.
\end{equation}

A set-valued tableau where every cell contains a set of
cardinality one is a semi-standard tableau of the same
shape $\lambda$ and weight $\alpha$.
This occurs iff $|\alpha|=|\lambda|$
and otherwise $|\alpha|>|\lambda|$.
If $\mathcal K_{\lambda\alpha}$ enumerates the set-valued tableaux 
of shape $\lambda$ and weight $\alpha$,  then
\begin{equation}
\label{grothins}
G_\lambda 
\,=\, \sum_{|\mu|\geq |\lambda|} (-1)^{|\lambda|+|\mu|}
\,
\mathcal K_{\lambda\mu}\, m_\mu 
\,=\, s_\lambda + \text{terms of higher degree}\,.
\end{equation}

{\it Dual Grothendieck polynomials} $g_\lambda$
(e.g.  \cite{[Lenart],[SZ],[LP]}) can be defined by 
$$
\langle g_\lambda, G_\mu\rangle = \delta_{\lambda\mu}
\,.
$$
Duality and \eqref{grothins} imply that
\begin{equation}
\label{defdualg}
h_\mu = \sum_{|\lambda|\leq |\mu|}
(-1)^{|\lambda|+|\mu|}\,
\mathcal K_{\lambda\mu}\, g_\lambda
\,.
\end{equation}
In fact, since the transition matrix $||\mathcal K ||_{\lambda\mu}$ 
is unitriangular, the system obtained from this expression over all partitions 
can be inverted and used to characterize the $\{g_\lambda\}$.  
Inverting \eqref{defdualg} also implies that
\begin{equation}
\label{dualgins}
g_\lambda =  s_\lambda + \text{terms of lower degree}
\end{equation}
by the triangularity of $h_\mu$ in terms of Schur functions.

\subsection{$k$-tableaux}

Let $\gg$ be a $k+1$-core and
let $\aa=(\aa_1,\ldots,\aa_r)$ be a composition of $|\kbnd(\gg)|$.
A \dit{$k$-tableau} of shape $\gg$ and weight $\aa$ 
is  a semi-standard filling of $\gg$ with integers $1,2,\ldots,r$ such that
the collection of cells filled with letter $i$ are labeled by exactly 
$\alpha_i$ distinct $k+1$-residues.

\smallskip

\begin{example}
\label{exssktab}
The $3$-tableaux of weight $(1,3,1,2,1,1)$ and shape (8,5,2,1) are:
\begin{equation}
{\tiny{\tableau*[scY]{5\cr 4&6\cr2&3&4&4&6\cr 1&2&2&2&3&4&4&6 }}} \quad
{\tiny{\tableau*[scY]{6\cr 4&5\cr 2&3&4&4&5\cr 1&2&2&2&3&4&4&5 }}} \quad
{\tiny{\tableau*[scY]{4\cr 3&6\cr 2&4&4&5&6\cr 1&2&2&2&4&4&5&6 }}}
\end{equation}
\end{example}

\begin{remark}
\label{ktabtab}
When $k\geq h(\gamma)$, a $k$-tableau $T$ of shape $\gamma$ and weight 
$\alpha$ is a semi-standard tableau of weight $\mu$ since no two diagonals 
of $T$ can have the same residue. 
\end{remark}

The symmetric family of
{\it dual $k$-Schur functions} was introduced in \cite{[LMhecke]} 
and defined to be the weight generating function of $k$-tableaux:
for any $k$-bounded partition $\lambda$,
\begin{equation}
\label{dualkschur}
\mathfrak S_\lambda^{(k)} = \sum_{T~k-tab\atop \shape(T)=\core(\lambda)} 
x^{w(T)}
\,.
\end{equation}
It is shown in \cite{[LMcore]} that 
the number $K_{\mu\alpha}^{(k)}$ 
of $k$-tableaux of shape $\core(\mu)$ and weight $\aa$ 
satisfies the property 
\begin{equation}
\label{unikkostka}
K_{\mu\lambda}^{(k)}=0 \quad\text{when}\quad \mu \ntrianglerighteq
\lambda \qquad
\text{ and }\qquad K_{\mu\mu}^{(k)}=1\,,
\end{equation}
for any $\lambda,\mu\in\mathcal P^k$.
Therefore, the monomial expansion has the form
\begin{equation}
\label{dualksenm}
\mathfrak S_\lambda^{(k)} = 
m_\lambda +
\sum_{\mu\in\mathcal P^k
\atop \mu \lhd\lambda}
K_{\lambda\mu}^{(k)}\, m_\mu\,,
\end{equation}
revealing that
$\{\mathfrak S_\lambda^{(k)}\}_{\lambda\in\mathcal P^k}$
forms a basis for
$$
\Lambda/
\mathcal I^{k} \quad\text{where}\quad 
\mathcal I^k=\langle m_\lambda : \lambda_1>k\rangle\,.
$$

\smallskip

This space is natural paired with
$$
\Lambda^{(k)}
= \mathbb Z[h_1,h_2,\ldots,h_k]\,.
$$
Since $\langle h_i: i>k\rangle$ is dual to $\mathcal I^k$
with respect to the Hall inner product,
$\Lambda^{(k)}$ is dual to to $\Lambda/\mathcal I^k$.
The basis for $\Lambda^{(k)}$ that is dual to 
$\{\mathfrak S_\lambda^{(k)}\}_{\lambda\in\mathcal P^k}$
is made up of the {\it $k$-Schur functions} $s_\lambda^{(k)}$.
Since the matrix $||K^{(k)}||_{\lambda,\mu\in \mathcal P^k}$ is invertible, 
the system:
\begin{equation}
\label{kschurdef}
h_\lambda = s_\lambda^{(k)}+\sum_{\mu : \mu\rhd\lambda} K_{\mu\lambda}^{(k)}
s_\mu^{(k)}\,\quad\text{for all } \lambda_1\leq k
\end{equation}
can be taken as the definition of $k$-Schur functions \cite{[LMproofs]}.

\section{Affine set-valued tableaux}
\label{astab}

\subsection{Definition}

In this section, we introduce and derive properties for a family of 
tableaux that generalizes both $k$-tableaux and set-valued tableaux.
In subsequent sections, from these tableaux we will extract 
an inhomogeneous generalization of (dual) $k$-Schur functions 
{\it and} an affine analog of (dual) Grothendieck polynomials.

\smallskip

Let $T_{\leq x}$ denote the subtableau obtained by deleting all
letters larger than $x$ from $T$.
For example,
$$
T= {\footnotesize\tableau*[mcY]
{\{7\}\cr\tiny \{2,5\}&\{6\}\cr\{ 1\}&\{2,3\}&\{4\}&\{4,6\} }}
\qquad
T_{\leq 4} =
{\footnotesize\tableau*[mcY]{\{2\}\cr \{1\}&\{2,3\}&\{4\}&\{4\} }}
$$

\begin{definition}
\label{stdsetvalued}
A standard affine set-valued tableau $T$ of degree $n$
is a set-valued filling such that, for each $1\leq x\leq n$,
$\shape(T_{\leq x})$ is a core and there is an $x$ exactly in 
all removable corners of $T_{\leq x}$ with the same residue.
\end{definition}

\begin{example}
\label{exstdsvktab}
With $k=2$, the standard affine s-v tableaux of degree 5 with shape 
$\core(2,1,1)=(3,1,1)$ are
\begin{equation}
{\tiny
{{\tableau*[lcY]{\{3,4\}_1\cr\{2\}_2\cr 
\{1\}_0&\{3,4\}_1&\{5\}_2 }}} \quad
{{\tableau*[lcY]{\{3,5\}_1\cr\{2\}_2\cr 
\{1\}_0&\{3\}_1&\{4\}_2 }}} \quad
{{\tableau*[lcY]{\{5\}_1\cr\{4\}_2\cr 
\{1,2\}_0&\{3\}_1&\{4\}_2 }}} \quad
{{\tableau*[lcY]{\{4\}_1\cr\{3\}_2\cr 
\{1,2\}_0&\{4\}_1&\{5\}_2 }}} \quad
{{\tableau*[lcY]{\{4\}_1\cr\{2,3\}_2\cr 
\{1\}_0&\{4\}_1&\{5\}_2 }}}}
\end{equation}
\begin{equation}
\label{ssasvtab}
{\tiny
{{\tableau*[lcY]{\{5\}_1\cr\{4\}_2\cr 
\{1\}_0&\{2,3\}_1&\{4\}_2 }}} \quad
{{\tableau*[lcY]{\{5\}_1\cr\{3,4\}_2\cr 
\{1\}_0&\{2\}_1&\{3,4\}_2 }}} \quad
{{\tableau*[lcY]{\{4\}_1\cr\{3\}_2\cr 
\{1\}_0&\{2\}_1&\{3,5\}_2 }}} }
\end{equation}
\end{example}

\smallskip

To define the semi-standard case of affine s-v tableaux,
we impose certain conditions on certain reading words of 
standard affine s-v tableaux. 
Recall that the reading word of a semi-standard tableau 
is made of the letters from top to bottom and left to right.
Semi-standard tableaux of weight $\alpha$ are
in bijection with standard tableaux having increasing reading 
words in the alphabets
\begin{equation}
\mathcal A_{\alpha,x} =
[1+\Sigma^{x-1}\alpha,\Sigma^x\alpha] \quad \text{where} \quad
\Sigma^x\alpha
 = \sum_{i\leq x}\alpha_i\,,
\end{equation}
for $x=1,\ldots,\ell(\alpha)$.  
This arises through {\it $\alpha$-standardization} defined 
on $T$ iteratively as follows:
start with $r=|\alpha|$ and relabel 
the rightmost $\ell(\alpha)$ by $r$.
Let $r=r-1$ and repeat. 
When there are no $\ell(\alpha)$ remaining, 
perform the relabeling of $\ell(\alpha)-1$.
$\alpha$-standardization applies to set-valued tableaux as well,
where the reading word is defined as usual and letters in the 
same cell are read in decreasing order.

\smallskip

It is in this spirit that we have defined our affine $K$-theoretic
generalization of tableaux.    Since letters in standard 
affine s-v tableaux 
can occur with multiplicity, we must consider
the {\it lowest reading word} in $\mathcal A$, obtained by
reading the lowest occurrence of the letters in $\mathcal A$ 
from top to bottom and left to right.  Again,
letters in the same cell are read in decreasing order.
In Example~\ref{exstdsvktab},the
lowest reading words in $\{1,\ldots,5\}$ are
$21435,52134,52134,32145,32145,51324,51243,41253$.

\smallskip

\begin{definition}
\label{setvalued}
For any $k$-bounded composition $\alpha$, an affine s-v tableau of 
weight $\alpha$ is a standard affine s-v tableau of degree $|\alpha|$ 
where, for each $1\leq x\leq \ell(\alpha)$,
\begin{enumerate}
\item the lowest reading word in 
$\mathcal A_{\alpha,x}$ is increasing 
\item the letters of $\mathcal A_{\alpha,x}$ occupy $\alpha_x$ distinct
residues
\item the letters of $\mathcal A_{\alpha,x}$ form a horizontal strip.
\end{enumerate}
\end{definition}

Let $\mathcal T^k_\alpha(\lambda)$ denote the
affine s-v tableaux of shape $\core(\lambda)$
and weight $\alpha$ and
$\mathcal T^k(\lambda)$ be those of shape $\core(\lambda)$ and
any weight.
Note the definition extends simply to skew affine s-v tableaux by 
deleting letters $1,\ldots,a$ from an affine s-v tableau.  

\begin{example}
The set $\mathcal T^2_{(2,1,1,1)}(2,1,1)$ of affine s-v tableaux
with weight (2,1,1,1) contains the affine s-v tableaux in \eqref{ssasvtab}
of Example~\ref{exstdsvktab}.
\end{example}

\smallskip

\subsection{Retrieving set-valued and $k$-tableaux}

To justify that affine s-v tableaux are in fact an affine 
$K$-theoretic version of tableaux, we connect them
to $k$-tableaux and set-valued tableaux.

\smallskip

\begin{proposition}
\label{propsvkt1}
For any  $k$-bounded partition $\lambda$ where $h(\lambda)\leq k$,
the affine  s-v tableaux of shape $\lambda$ are simply the set-valued 
tableaux of shape $\lambda$.
\end{proposition}
\begin{proof}
Given $h(\lambda)\leq k$, $\core(\lambda)=\lambda$ and no 
two diagonals of $\lambda$ have the same residue.  
Thus, an affine s-v tableau of weight $\alpha$ is a set-valued tableau 
with no repeated letters.  The substitution of all letters 
in $\mathcal A_{\alpha,x}$ by $x$ gives a set-valued tableau
of weight $\alpha$
since conditions on affine s-v tableaux imply these letters 
form a horizontal strip and occupy $\alpha_x$ distinct cells.
On the other hand, the $\alpha$-standardization of a set-valued 
tableau of weight $\alpha$ ensures that the reading word of
the $\alpha_x$ letters in $\mathcal A_{\alpha,x}$ is increasing
and forms a horizontal strip.
\end{proof}

To make the connection with $k$-tableaux, we need several basic 
properties of affine s-v tableaux.  We say that $x$ is {\it lonely}
when a letter $x$ occurs in cell without another letter. 

\begin{property}
\label{skewnonempty}
Given an affine s-v tableau $T$,
$\shape(T_{\leq x})/\shape(T_{\leq x-1})\neq \emptyset$
if and only if every $x$ is lonely in $T_{\leq x}$, for
any letter $x\in T$.
\end{property}
\begin{proof}
If $\shape(T_{\leq x})=\shape(T_{\leq x-1})$ then 
no $x$ can be lonely since $T_{\leq x-1}$ is obtained 
by deleting the letter $x$ from $T_{\leq x}$.
On the other hand,
we will show that if $T_{\leq x}$ has some $x$ that
is not lonely then no $x$ are lonely implying
the shapes must be equal. 
Suppose $T_{\leq x}$ has a cell of some residue $i$ 
containing $x$ by itself and one containing $x$ with another letter.
Then $\shape(T_{\leq x-1})$ has both a removable and an addable 
$i$-corner.
This contradicts that a core never contains an addable and removable
$i$-corner.
\end{proof}

\begin{property}
\label{completeset}
Given an affine s-v tableau $T$, if $\gamma^{(x)}=\shape(T_{\leq x})$ then
$$|\mathfrak p(\gamma^{(x)})|= |\mathfrak p(\gamma^{(x-1)})| +1
\quad\text{for any} \quad
\gamma^{(x)}\neq\gamma^{(x-1)}\,.
$$ 
\end{property}
\begin{proof}
Given $T$ is an affine s-v tableau, there is an
$x$ in all removable $i$-corners of the core 
$\gamma^{(x)}$.  If $\gamma^{(x)}\neq\gamma^{(x-1)}$
then these $x$ are all lonely by Property~\ref{skewnonempty}.
Therefore, $\gamma^{(x)}$ is $\gamma^{(x-1)}$ plus 
addable $i$-corners
and the result follows from
Remark~\ref{addonecore}.
\end{proof}

From Property~\ref{completeset}, an affine s-v tableau $T$ of weight $\alpha$ 
satisfies $|\kbnd(\shape(T_{\leq x}))|=|\kbnd(\shape(T_{\leq x-1}))|+\{0,1\}$ 
for all $x=1,2,\ldots,|\alpha|$.  Consequently:

\begin{corollary}
\label{tridegree}
If there is an affine s-v tableau of weight
$\alpha$ and shape $\core(\lambda)$, then
$|\alpha|\geq |\lambda|$.
\end{corollary}

\smallskip

We are now prepared to prove that the family of affine s-v tableaux 
includes $k$-tableaux.

\begin{proposition}
\label{propsvkt2}
The set of affine s-v tableaux
$\mathcal T_\alpha^k(\lambda)$ when $|\alpha|=|\lambda|$ 
is the set of $k$-tableaux with weight $\alpha$ and shape $\core(\lambda)$.
\end{proposition}
\begin{proof}
Consider an affine s-v tableau $T$ of weight $\alpha$ and
shape $\gamma=\core(\lambda)$ where $n=|\alpha|=|\lambda|$.  
For $x=1,\ldots, n$, let $\lambda^{(x)}=\kbnd(\shape(T_{\leq x}))$.
Since $|\lambda^{(x)}|=|\lambda^{(x-1)}|+\{0,1\}$
by Property~\ref{completeset} and $|\lambda^{(n)}|=n$,
$|\lambda^{(x)}|=|\lambda^{(x-1)}|+1$ for all $x$.
Thus, 
$\lambda^{(x-1)}\neq\lambda^{(x)}$ 
which implies that all $x$ are lonely in $T_{\leq x}$ by 
Property~\ref{skewnonempty}, for all $x$. 
Therefore no cell of $T$ has a set of cardinality more than one.
The conditions on affine s-v tableau imply that
replacing all letters in $\mathcal A_{\alpha,x}$ by $x$
gives a semi-standard filling where $x$ occupies $\alpha_x$
distinct residues and thus, $T$ is a $k$-tableau.

On the other hand, given a $k$-tableau $T$ of weight $\alpha$, 
we can {$\alpha$-standardize} $T$ iteratively from $r=|\alpha|$ as 
follows: relabel every $x(i)\in T$ by $r$, where $x$ is the rightmost 
letter of $T$ that is not larger than $r$ and $i$ is its residue.  
Let $r=r-1$.  It was shown in \cite{[LMcore]} that the resulting 
tableau $U$ is a standard $k$-tableau and by construction, the lowest 
reading 
word in $\mathcal A_{\alpha,x}$ is clearly increasing.  Therefore, 
since the letter $x$ occupies $\alpha_x$ distinct residues and forms 
a horizontal strip in $T$, $U$ meets the conditions of 
an affine s-v tableau of weight $\alpha$.
\end{proof}

\section{Alternate characterizations}

\subsection{Affine Weyl group characterization}

In the theory of $k$-Schur functions, the discovery that 
$k$-tableaux are reduced words for grassmannian 
permutations in the affine symmetric group was the spring board
to understanding the $k$-Schur role in geometry.  Here, we
investigate a similar interpretation for affine s-v tableaux.
To start, we consider the standard case and recall results in 
the $k$-tableaux case.

\smallskip

Let $\tilde S_{n}$ denote the affine Weyl group of $A_{n-1}$,
generated by $\langle s_0, s_1,\ldots, s_{n-1}\rangle$ and
satisfying the relations
\begin{equation}
\label{coxeter}
\begin{array}{ccc}
&s_i^2 = 1 & \text{ for all $i$ } \\
&s_is_{i+1}s_i = s_{i+1}s_is_{i+1}& \text{for all $i$}\\
&s_is_j = s_js_i & \text{if}\; |i-j|> 1\,,
\end{array}
\end{equation}
where indices are taken modulo $n$ (hereafter we will always be 
working mod $n$).  A word $i_1 i_2\cdots i_m$ in the alphabet 
$\{0,1,\ldots,n-1\}$ corresponds to the permutation $w \in \tilde S_{n}$ 
if $w=s_{i_1} \dots s_{i_m}$.  
The length $\ell(w)$  of $w$  is defined to be the length of its
shortest word.  Any word of this length is said to be {\it reduced}
and we denote the set of all reduced words for $w$ by
$\mathcal R(w)$.
The set $\tilde S_n^0$ of grassmannian elements are the minimal 
length coset representatives of $\tilde S_n/S_n$, where $S_n$ is the 
finite symmetric group. In fact, $w\in \tilde S_n^0$ iff 
every reduced word for $w$ ends in $0$.  

\smallskip

Consider operators on set-valued tableaux defined
for $i=0,\ldots,k$ by
$$\mathfrak s_{i,x}: \hat T \to T \,,
$$ 
where $T$ is obtained by adding an $x$ to
all addable or all removable $i$-corners of $\hat T$.  
It turns out \cite{[LMcore]} that the set of 
reduced words for a fixed $w\in\tilde S_{k+1}^0$
is in bijection with the set of $k$-tableaux of some
shape $\core(\lambda)$.  In particular, each reduced word 
$i_\ell i_{\ell-1} \cdots i_1$ for $w\in \tilde S^{0}_{k+1}$ 
is sent to the standard $k$-tableau
$\s_{i_\ell,\ell} \s_{i_{\ell-1},\ell-1}\cdots \s_{i_1,1}\emptyset$
on $\ell$ letters.

\smallskip

It is natural to work with operators defined on shapes,  
for $i=0,\ldots,k$, by $$
\mathfrak s_{i}: \gamma \to \gamma\; +\;
\text{its addable $i$-corners} \,.
$$
These can be viewed as an affine analog of operators 
introduced in \cite{Buch} and where the set-up follows that 
of \cite{FG}.

\begin{remark}
\label{addonezerosi}
A number of useful properties are satisfied by the
$\mathfrak s_i$ operators.
\begin{enumerate}
\item If $\beta$ is a core with an addable $i$-corner, then
$\s_i(\beta)$ is a core and
$|\kbnd(\s_i(\beta))|=|\kbnd(\beta)|+1$
by Remark~\ref{addonecore}
\item
The bijection between $\tilde S_{k+1}^0$ and $\mathcal C^{k+1}$
defined by the map
$$
\mathfrak t:\;  w\; \rightarrow\;
\core(\lambda)=\s_{i_\ell} \cdots \s_{i_1}\emptyset
\,,
$$
for any $i_\ell\cdots i_1\in \mathcal R(w)$,
has the property that $|\lambda|=\ell(w)$
since a word of length $\ell$ corresponds to a 
$k$-tableaux on $\ell$ letters.
\item
If $i_\ell i_{\ell-1}\cdots i_1$ is a reduced word for 
an affine grassmannian permutation,
then $\s_{i_\ell}\s_{i_{\ell-1}} \cdots \s_{i_1}\emptyset$ has an
addable $i_j$-corner for all $j<\ell$ by (1) and (2).
\item  When acting on cores, we have the relations
$$
\s_i^2=\s_i\;\text{for all $i$},\quad
\s_i\s_{i+1}\s_i = \s_{i+1}\s_i\s_{i+1},\quad
\s_i\s_j=\s_j\s_i\;\text{when $|i-j|>1$}\,.
$$
\end{enumerate}
\end{remark}

Denote the affine grassmannian permutation associated 
to $\lambda\in\mathcal P^k$ by
$$
w_\lambda=\mathfrak t^{-1}(\core(\lambda))
\,.
$$
\begin{remark}
\label{part2word}
A quick way to construct $w_\lambda$ from $\lambda$ is to take the 
residues of $\lambda$  read from {\it right} to {\it left}
and top to bottom \cite{[LMcore]}.  For example,
$w_{(2,1,1)}\in \tilde S_3$ is
\begin{equation}
{\tiny{\tableau*[scY]{1\cr 2\cr 
0&1}}} \quad
\to 1\,2\,1\,0
\end{equation}
\end{remark}

The relations in Remark~\ref{addonezerosi}(4) arise
in the nil Hecke algebra.  We will discuss this
further in \S~\ref{AGP}, but for now are interested in 
studying the equivalence classes of words under these relations.
To be precise, we consider the set $\mathcal W(w_\lambda)$ 
of all words whose reduced expression is in $\mathcal R(w_\lambda)$
under the relations
\begin{equation}
\begin{array}{ccc}
&u_i^2=u_i & \quad\text{for all $i$}\\
&u_i u_{i+1} u_i = u_{i+1} u_i u_{i+1}&\quad\text{for all $i$}\\
&u_i u_j= u_j u_i&\quad\text{when $|i-j|>1$}\,
\end{array}
\end{equation}
Note that
$\mathcal R(w_\lambda)= 
\{ u\in \mathcal W(w_\lambda) : \ell(u)=\ell(w_\lambda)\}$.

\begin{remark}
\label{hasaddableorremovable}
For any $i_r\cdots i_1\in \mathcal W(w_\lambda)$,
$\s_{i_{r}}\cdots \s_{i_1}\emptyset$ has a removable or
an addable $i_m$-corner for all $m\leq r$.
In particular, $i_r\cdots i_1\in \mathcal W(w_\lambda)$ if and only if $i_1=0$.
This follow by iterating the following argument:
if $i_{\ell+1}i_\ell \cdots i_1\in\mathcal W(w_\mu)$ then
there is some $t$ where 
$j_{\ell}\cdots j_t j_t\cdots  j_1\sim 
i_{\ell+1} i_{\ell}\cdots i_1$, for $\ell=\ell(w_\mu)$.
Then $\s_{j_\ell}\cdots \s_{j_1}\emptyset$ has an
addable $j_m$-corner by Remark~\ref{addonezerosi}(3) implying
$\s_{j_\ell}\cdots\s_{j_t}\s_{j_t}\cdots \s_{j_1}\emptyset$ has an
addable or removable  $j_m$-corner.
Therefore, $\s_{i_{\ell+1}}\cdots \s_{i_1}\emptyset$
has an addable or a removable $i_m$-corner for any $m\leq \ell+1$
by Remark~\ref{addonezerosi}(4).
\end{remark}

\smallskip

Just as the $k$-tableaux represent reduced words for
affine grassmannian permutations, we find that standard
affine s-v tableaux of fixed shape $\core(\lambda)$ are 
none other than the words whose reduced expression is
in $\mathcal R(w_\lambda)$ .

\begin{proposition}
\label{bij}
For $\lambda\in\mathcal P^k$, there is a bijection
$$
\mathfrak s:\mathcal W({w_\lambda})\to \mathcal T^{k}_{1^{m}}(\lambda)
$$
defined by $\mathfrak s(i_m i_{m-1} \cdots i_1)=
\s_{i_m,m} \s_{i_{m-1},m-1}\cdots s_{i_1,1}\emptyset$.
\end{proposition}
\begin{proof}
Given ${i_m}\cdots {i_1}\in\mathcal W(w_\lambda)$,
we claim that $T=\mathfrak s(i_{m}\cdots i_1)$
is a standard affine s-v tableau of degree $m$.
Since $T_{\leq 1}= {\tiny\tableau*[mcY]{\{1\}} }$,
assume by induction that $T_{\leq m-1}$ is a standard
affine s-v tableau of core shape $\gamma$.
Remark~\ref{hasaddableorremovable} implies $\gamma$
has either an addable or a removable $i_{m}$-corner
and thus $T=T_{\leq m}$ is 
obtained by putting an $m$ in all such corners.
Therefore the shape of $T$ is either $\gamma$ or
$\gamma$ plus its addable $i_{m}$-corners;
in both cases it is a core.
Further, $T$ has shape $\core(\lambda)$ by 
Remark~\ref{addonezerosi}.

For any $T\in \mathcal T^k_{1^m}(\lambda)$,
$T=\s_{j_m,m}\cdots \s_{j_1,1}\emptyset$ where
$j_a$ denotes the residue of letter $a$ in $T$.
Thus, $\core(\lambda)=\s_{j_m}\cdots \s_{j_1}\emptyset$ and
to prove $\mathfrak s$ is onto, we need
${j_m}\cdots {j_1}\in\mathcal W(w_\lambda)$.
Consider a reduced word $j_{\ell}'\cdots j_1'$ 
equivalent to ${j_m}\cdots {j_1}$.
Remark~\ref{addonezerosi}(4) implies
$\core(\lambda)= \s_{j_{\ell}'}\cdots \s_{j_1'}\emptyset$
and therefore
${j_{\ell}'}\cdots {j_1'}\in\mathcal R(w_\lambda)$.

To see that $\mathfrak s$ is 1-1, consider $\mathfrak s(i_m\cdots i_1)=
\mathfrak s(\hat i_{\hat m}\cdots \hat i_{1})$
for $i_m\cdots i_{1}, \hat i_{\hat m}\cdots \hat i_{1}\in
\mathcal W(w_\lambda)$. By definition, $\mathfrak s_{i,x} \hat T$ 
has an $x$ in the addable or removable $i$-corners of $\hat T$.  
Therefore, 
if there is some (minimal) $j$ where $i_j\neq \hat i_j$ clearly
$\mathfrak s_{i_j,j} \hat T \neq\mathfrak s_{\hat i_j,j} \hat T$.
\end{proof}

\subsection{Horizontal strip characterization}

In classical and $k$-tableaux theory, characterizing tableaux 
as a sequence of shapes satisfying horizontality conditions 
has many applications.  Most notably, the Pieri rule for affine 
and usual Schur functions is readily apparent in this context.
Moreover, given such an interpretation, the tie between the affine 
Weyl group and $k$-tableaux in the semi-standard case can be made.
With this in mind, we set-out to coin affine s-v tableaux in 
terms of certain horizontal strips. 

\smallskip

We start with the characterization of set-valued 
and $k$-tableaux as a sequence of shapes.  
A {\it set-valued} $r$-strip $(\gamma/\beta,\rho)$ is 
such that
\begin{enumerate}
\item $\gamma/\rho$ is a horizontal $r$-strip 
\item $\beta/\rho$ is a set of $r-|\gamma/\beta|$ 
$\beta$-removable corners
\end{enumerate}
The set-valued tableaux of weight $\alpha$
and shape $\lambda$ are in bijection with the set of 
sequences having the following form:
$$
(\emptyset,\emptyset)\subset (\lambda^{(1)},\rho^{(1)})
\subseteq (\lambda^{(2)},\rho^{(2)})
\subseteq\cdots
\subseteq (\lambda^{(\ell(\alpha))},\rho^{(\ell(\alpha))})
=(\lambda,\rho^{(\ell(\alpha))})
$$
where $(\lambda^{(x)}/\lambda^{(x-1)},\rho^{(x)})$ is a set-valued 
$\alpha_x$-strip, for $x=1,\ldots, \ell(\alpha)$.

\smallskip

For $0\leq r\leq k$, an {\it affine $r$-strip}
$\gamma/\beta$ is a horizontal strip where
\begin{description}
\item[af1] $\gamma$ and $\beta$ are cores
\item[af2] $|\kbnd(\gamma)|-|\kbnd(\beta)|=r$
\item[af3] $\gamma/\beta$ occupies $r$ distinct residues. 
\end{description}
It can be deduced from results in \cite{[LMcore]} that  $k$-tableaux
of weight $\alpha$ and shape $\gamma$ are in bijection with 
chains of the form:
\begin{equation}
\label{altdefktab}
\emptyset\subset
\gamma^{(1)}\subseteq \gamma^{(2)}\subseteq
\cdots\subseteq \gamma^{(\ell(\alpha))}=\gamma
\,,
\end{equation}
where $\gamma^{(x)}/\gamma^{(x-1)}$ is an affine 
$\alpha_{x}$-strip, for $x=1,\ldots, \ell(\alpha)$.

\smallskip

We impose additional conditions on these strips in order
to view affine s-v tableaux in a similar way. 
Given cores $\beta\subseteq\gamma$, a cell
of $\beta$ that lies below a cell of $\gamma$ is 
{\it $\gamma$-blocked}.

\begin{definition}
\label{incstrip}
For $0\leq r\leq k$,
an \dit{affine set-valued} $r$-strip $(\gamma/\beta,\rho)$ is 
such that

\begin{description}
\item[asv1]  $\gamma/\rho$ is a horizontal strip

\item[asv2] $\gamma/\beta$ is an affine horizontal $r-m$-strip,
where $m=|Res(\beta/\rho)|$.

\item[asv3] $\beta/\rho$ is a subset of $\beta$-removable corners 
such that if $c(i)\in\beta/\rho$ then all $\beta$-removable
$i$-corners that are not $\gamma$-blocked lie in $\beta/\rho$.
\end{description}
\end{definition}

\begin{remark}
\label{stripreduce}
A priori, if $(\gamma/\beta,\rho)$ is an affine s-v $r$-strip 
then $\gamma/\beta$ is an affine strip.  In fact, 
$\gamma/\beta$ is an affine $r$-strip when $\beta=\rho$. 
When $k=\infty$, an affine s-v strip is simply a set-valued
strip since residues can occur at most once in a horizontal strip.
\end{remark}

To prove that affine s-v tableaux can be characterized 
in terms of affine s-v strips, we need a number of 
properties about these and affine strips.

\begin{proposition}
\label{bigthe}
\label{iteratestrip}
If $\gamma/\beta$ is an affine $r$-strip and $i$ is the
residue of its rightmost cell, then $\hat\gamma/\beta$ is
an affine $r-1$-strip where $\hat\gamma$ is
$\gamma$ minus its removable $i$-corners.
\end{proposition}
\begin{proof}
Since $|\kbnd(\hat\gamma)|=|\kbnd(\gamma)|-1$ by 
Remark~\ref{addonecore}, it suffices to prove that 
$\beta\subseteq\hat\gamma$.  To this end, we will prove that
the $\gamma$-removable $i$-corners 
are exactly the cells of residue $i$ in $\gamma/\beta$,
where $i$ is the residue of the rightmost cell in $\gamma/\beta$.
Since $\gamma/\beta$ is a horizontal strip 
whose rightmost cell has residue $i$ then all cells of residue 
$i$ in $\gamma/\beta$ are $\gamma$-removable by Property~\ref{thecoreprop} 
since they are all extremal in $\gamma$.  It thus remains to show that any 
$\gamma$-removable $i$-corner is in $\gamma/\beta$.

Suppose by contradiction there is a $\gamma$-removable
$c(i)$ that is also $\beta$-removable (choose the lowest).
Note that it lies higher than any 
$a(i)\in\gamma/\beta$ since the $z(i+1)$ beneath
$a$ is at the top of its column in $\beta$ implying
that all extremals of residue $i+1$ lower than $a$
are at the top of their column in $\beta$ by Property~\ref{thecoreprop}.
Then, the cell left-adj to any $a$ must be in $\gamma/\beta$ since 
otherwise it would end its row in $\beta$ whereas the
extremal left-adj to $c(i)$ does not (contradicting 
Property~\ref{thecoreprop}).
Therefore, the column with $c$ has more $k$-bounded hooks
in $\beta$ than in $\gamma$.
Let $a_1,\ldots, a_r$ denote the lowest cells of residues 
$i_1,\ldots,i_r$, respectively, in $\gamma/\beta$.
Each of these cells lies in a column with one more
$k$-bounded hook in $\gamma$ than in $\beta$
by Property~\ref{thecoreprop}.
In fact, these are the only columns where $\gamma$ has more
$k$-bounded hooks than $\beta$.
Since $|\kbnd(\gamma)|-|\kbnd(\beta)|=r$ it must be that
no column of $\beta$ has more $k$-bounded hooks than
in $\gamma$ and we have our contradiction.
\end{proof}

\begin{property}
\label{distinctres}
Let $\gamma/\beta$ be an affine $r$-strip for some $r\leq k$.
If $\beta$ has a removable $i$-corner that is not $\gamma$-blocked
then $i\not\in Res(\gamma/\beta)$.
\end{property}
\begin{proof}
Consider cell $c(i)$, in some row $r_c$, 
that is $\beta$-removable and at the
top of its column in $\gamma$. Suppose 
by contradiction that there is a $y(i)\in
\gamma/\beta$ and let $r_y$ be the lowest row containing
such a cell.
With $\gamma=\gamma^{(0)}$, let $\gamma^{(j)}$ be the core 
obtained by deleting all $i_j$-corners of $\gamma^{(j-1)}$, 
where $i_j$ is the residue of the rightmost element in 
$\gamma^{(j-1)}/\beta$.  
Proposition~\ref{iteratestrip} implies that
$\gamma^{(j)}/\beta$ is an affine strip.

Note that $y(i)\in\gamma/\beta$ implies that there is some $t$ where 
$i_t=i$.  Since $y$ is the rightmost element of $\gamma^{(t-1)}/\beta$
and all cells  in the same row and to the right of a $\beta$-removable 
cell are in
$\gamma/\beta$, $c(i)$ lies at the end of its row in 
$\gamma^{(t-1)}$ (and thus in $\gamma^{(t)}$) if $r_c<r_y$.
In this case, the cell below $c(i)$ is extremal in 
$\gamma^{(t)}$.  However, 
the cell below $y(i)$ lies at the top of its column in $\gamma^{(t)}$
violating Property~\ref{thecoreprop}.  Therefore, $r_c\geq r_y$.
Note that the cell $z(i-1)$ left-adj to $y(i)$ lies at the end of 
row $r_y$ in $\gamma^{(t)}$.  Thus it is above an extremal
cell of residue $i$ whereas $c(i)$ is at the top of its column.
Again by Property~\ref{thecoreprop} we have a contradiction.
\end{proof}

\smallskip

Since $\gamma/\rho$ is horizontal for any affine s-v strip 
$(\gamma/\beta,\rho)$, no cell in $\beta/\rho$ is
$\gamma$-blocked.  We therefore deduce that:

\smallskip

\begin{corollary}
\label{distinctrescor1}
For any affine set-valued $r$-strip $(\gamma/\beta,\rho)$,
$Res(\gamma/\beta)\cap Res(\beta/\rho)=\emptyset$.
\end{corollary}

\smallskip

In particular, since $|Res(\gamma/\beta)|=r-m$ by the definition
of affine strip, we have:

\smallskip

\begin{corollary}
\label{distinctrescor}
For any affine set-valued $r$-strip $(\gamma/\beta,\rho)$,
$|Res(\gamma/\rho)|=r$.
\end{corollary}

\begin{property}
\label{iteratesvlemma}
If $(\gamma/\beta,\rho)$ is an affine s-v strip and
$i$ is the residue of the rightmost cell in $\gamma/\rho$
then the cells of residue $i$ in $\gamma/\rho$ 
are exactly the $\gamma$-removable $i$-corners.
\end{property}
\begin{proof}
If the rightmost cell $c(i)$ of $\gamma/\rho$ is in $\gamma/\beta$,
Property~\ref{bigthe} implies that the $\gamma$-removable $i$-corners
are exactly the cells of $\gamma/\beta$ with residue $i$, which
are exactly the cells of $\gamma/\rho$ with residue $i$ by
Corollary~\ref{distinctrescor1}.  
If $c(i)\in\beta/\rho$, note that all cells of residue $i$ in 
$\gamma/\rho$ are $\gamma$-removable corners since the lowest 
cell of residue $i$ in $\gamma/\rho$ is at the end of its row 
implying all $i$-extremals are at the end of their row in $\gamma$ 
by Property~\ref{thecoreprop}.  If there is a $\gamma$-removable 
corner $\bar c(i)\not\in\gamma/\rho$, then
$\bar c\in\rho\subseteq\beta\subseteq\gamma$ implies 
$\bar c\not\in\beta/\rho$ and is a $\beta$-removable $i$-corner.
Therefore $\bar c$ is $\gamma$-blocked by definition of affine 
s-v strip, contradicting that $\bar c$ is $\gamma$-removable.
\end{proof}

\begin{proposition}
\label{iteratesvstrip}
Consider the affine s-v strip $(\gamma/\beta,\rho)$ 
whose rightmost cell in $\gamma/\rho$ has residue $i$.
If $i\in Res(\gamma/\beta)$ then
$(\hat\gamma/\beta,\rho)$ is an affine s-v strip where
$\hat\gamma$ is $\gamma$ minus its $i$-corners.
Otherwise, $(\gamma/\beta,\hat\rho)$ is an affine s-v strip where
$\hat\rho$ is $\rho$ plus its $i$-corners.
\end{proposition}
\begin{proof}
Let $c(i)$ be the rightmost cell of $\gamma/\rho$.
If $c(i)\in\beta/\rho$ then $i\not\in Res(\gamma/\beta)$
by Corollary~\ref{distinctrescor1} and there are no 
$\gamma$-addable $i$-corners since
$c(i)$ is $\gamma$-removable by Property~\ref{iteratesvlemma} and 
$\gamma$ is a core.  Therefore, all addable $i$-corners of $\rho$ 
are in $\beta/\rho$ implying $\hat\rho\subseteq\beta$ and
thus that $(\gamma/\beta,\hat\rho)$ is
an affine s-v strip.

If $c(i)\in\gamma/\beta$ then Proposition~\ref{iteratestrip} implies
(asv1) and (asv2).
Now suppose there are $\beta$-removables, 
$c_1(j)\in\beta/\rho$ and $c_2(j)\not\in\beta/\rho$.
Since $(\gamma/\beta,\rho)$ is an affine s-v strip,
$c_2(j)$ must be $\gamma$-blocked. If, by 
contradiction, $c_2$ is not $\hat\gamma$-blocked
then $j=i+1$.
$c(i)$ lies at the end of its row in $\gamma$
implying all higher $i$-extremals are at the end of their row 
in $\gamma$ by Property~\ref{thecoreprop}.
However, by horizontality of $\gamma/\rho$, 
the cell left-adj to $c_1$ is an $i$-extremal.
\end{proof}

Now we are equipped to rephrase the definition of affine s-v tableaux.
For $\lambda\in\mathcal P^k$,
consider the set of pairs
obtained by adding affine s-v strips to $\core(\lambda)$,
$$
\mathcal H^{k}_{\lambda,r} = 
\{ (\mu,\rho): (\core(\mu)/\core(\lambda),\rho)
=\text{affine~set-valued~$r$-strip}
\}
\,.
$$

\begin{theorem}
\label{setvalued2}
For any $\lambda\in\mathcal P^k$ and
$k$-bounded composition $\alpha$,
there is a bijection between
$\mathcal T^k_\alpha(\lambda)$ and
$$
\left\{
(\emptyset,\emptyset)\subset
(\gamma^{(1)},\rho^{(1)})
\subseteq 
\cdots\subseteq 
(\gamma^{(\ell(\alpha))},\rho^{(\ell(\alpha))})
: (\kbnd(\gamma^{(x)}),\rho^{(x)})
\in \mathcal H^k_{\mathfrak p(\gamma^{(x-1)}),\alpha_x}
\right\}\,,
$$
where $\gamma^{(0)}=\emptyset$ and $\gamma^{(\ell(\alpha))}=\core(\lambda)$.
\end{theorem}

\begin{proof}
($\Leftarrow$):
For $x=1,\ldots,\ell(\alpha)$, we will construct $T$ by filling the cells 
of $\gamma^{(x)}/\rho^{(x)}$ iteratively as follows: start with 
$N=\Sigma^x\alpha$.  Let $i$ denote the residue of the 
rightmost cell in $\gamma^{(x)}/\rho^{(x)}$ 
that contains no letter larger than $N$.
Put letter $N$ in all cells of $\gamma^{(x)}/\rho^{(x)}$ with
residue $i$.
Let $N=N-1$ and repeat.

Since $\gamma^{(0)}=\emptyset$,
$\gamma^{(1)}$ is the row shape $(\alpha_1)$
and thus $T_{\leq \alpha_1}=
{\tiny{\tableau*[scY]{1&2&\cdots&\cdots &\alpha_1}}}$
is an affine s-v tableau of weight $(\alpha_1)$.
By induction, assume $T_{\leq \Sigma^{x-1}\alpha}$
is an affine s-v tableau of weight $(\alpha_1,\ldots,\alpha_{x-1})$.
Note that Conditions (1) and (3) on affine s-v tableaux
are met by construction since $\gamma^{(x)}/\rho^{(x)}$ is a horizontal 
strip.  Further, Corollary~\ref{distinctrescor} implies Condition (2).
It thus remains to show
that $T_{\leq\Sigma^x\alpha-j}$
is a set-valued tableau with $\Sigma^x\alpha-j$ in all removable 
corners of some residue and its shape $\tau^{(j)}$ is a core,
for $j=0,\ldots,\alpha_x-1$.

Let $\beta=\gamma^{(x-1)}$ and $\eta^{(0)}=\rho^{(x)}$.
Note that $\tau^{(0)}=\gamma^{(x)}$.  
By construction, there is 
an $N=\Sigma^x\alpha$ in all cells of $\tau^{(0)}/\eta^{(0)}$ with residue 
$i_0$, where $i_0$ is the residue of the rightmost cell $c$ in 
$\tau^{(0)}/\eta^{(0)}$.
Property~\ref{iteratesvlemma} implies there is a $N$ in 
exactly the $\tau^{(0)}$-removable $i_0$-corners.

Note that $\tau^{(1)}$ is $\tau^{(0)}$ minus cells containing a lonely $N$.
If $c\in\tau^{(0)}/\beta$ this is $\tau^{(0)}$ minus its $i_0$-corners
and otherwise $\tau^{(1)}=\tau^{(0)}$.  
In the later case, set $\eta^{(1)}=\eta^{(0)}$ plus its $i_0$-corners
and otherwise let $\eta^{(1)}=\eta^{(0)}$.
Proposition~\ref{iteratesvstrip} 
thus implies that $(\tau^{(1)}/\beta,\eta^{(1)})$ is an affine s-v strip 
and in particular, $\tau^{(1)}$ is a core.  By construction, there 
is a $N-1$ in all cells of $\tau^{(1)}/\eta^{(1)}$ with residue $i_1$, 
where $i_1$ is the residue of the rightmost cell
$c_1\in\tau^{(1)}/\eta^{(1)}$.
Thus, there is an $N-1$ in exactly the $\tau^{(1)}$-removable
$i_1$-corners by Property~\ref{iteratesvlemma}.
Iterating this argument proves the claim.

\smallskip

($\Rightarrow$)
Given $U\in \mathcal T^k_\alpha(\lambda)$,
let $\gamma^{(x)}=\shape(U_{\leq \Sigma^x\alpha})$ and $\rho^{(x)}$ 
be the shape obtained by deleting any {\it cell} containing an 
element of $\mathcal A_{\alpha,x}$ from $U_{\leq \Sigma^x\alpha}$,
for $x=1,\ldots,\ell(\alpha)$.
We claim that each $(\gamma^{(x)}/\gamma^{(x-1)},\rho^{(x)})$
is an affine set-valued $\alpha_x$-strip.

Let $T=U_{\leq \Sigma^x\alpha}$ and note that
\begin{itemize}
\item $\gamma^{(x)}/\rho^{(x)}$ are the cells in
$T$ containing a letter in $\mathcal A_{\alpha,x}$
\item $\gamma^{(x-1)}/\rho^{(x)}$ are the cells in
$T$ containing a letter in $\mathcal A_{\alpha,x}$ {\it and} 
a letter weakly smaller than $\Sigma^{x-1}\alpha$
\item $\gamma^{(x)}/\gamma^{(x-1)}$ are the 
cells in $T$ containing only letters larger than $\Sigma^{x-1}\alpha$.
\end{itemize}
The definition of affine s-v tableaux implies that
$\gamma^{(x)}/\rho^{(x)}$ is a horizontal strip and 
that $\gamma^{(x)}$ and $\gamma^{(x-1)}$ are cores.

To show that $\gamma^{(x)}/\gamma^{(x-1)}$ is
an affine $\alpha_x-|Res(\gamma^{(x-1)}/\rho^{(x)})|$-strip, 
we note it is a horizontal strip
since $\rho^{(x)}\subseteq\gamma^{(x-1)}$.
We claim that no letter lies in both
$\gamma^{(x)}/\gamma^{(x-1)}$ and $\gamma^{(x-1)}/\rho^{(x)}$
implying (af3)
since the $\alpha_x$ letters of $\mathcal A_{\alpha,x}$
each occupy a distinct residue and lie in $\gamma^{(x)}/\rho^{(x)}$.
To this end, suppose $y=a+\Sigma^{x-1}\alpha\in\gamma^{(x)}/\gamma^{(x-1)}$.
It must be lonely in $T_{\leq y}$
since $\gamma^{(x)}/\gamma^{(x-1)}$ has no letters
smaller than $\Sigma^{x-1}\alpha$ and no two letters of
$\mathcal A_{\alpha,x}$ share a cell.  The proof of
Property~\ref{skewnonempty} then implies that {\it all} $y$ 
are lonely in $T_{\leq y}$ and thus
{\it all} $y\in \gamma^{(x)}/\gamma^{(x-1)}$.
To verify (af2), let
$$
\tau^{(a)}=\shape(T_{\leq a+\Sigma^{x-1}\alpha})\,,
$$
for $a=1,\ldots, \alpha_x$.  Since any
$a+\Sigma^{x-1}\alpha\in\gamma^{(x)}/\gamma^{(x-1)}$
is lonely, 
$\tau^{(a)}\neq\tau^{(a-1)}$ for these $a$ and
$|\tau^{(a)}|=|\tau^{(a-1)}|+1$ by
Property~\ref{completeset}.
Otherwise, $|\tau^{(a)}|=|\tau^{(a-1)}|$
by definition of $\gamma^{(x-1)}/\rho^{(x)}$.

To prove Condition (asv3), note that
all cells of $\gamma^{(x-1)}/\rho^{(x)}$ are removable 
corners of $\gamma^{(x-1)}$ since
these cells contain both a letter larger and weakly 
smaller than $\Sigma^{x-1}\alpha$ and rows/columns are non-decreasing.
We thus must show that if
$c$ and $\bar c$ are removable $i$-corners of $\gamma^{(x-1)}$
where $c\in \gamma^{(x-1)}/\rho^{(x)}$ and
$\bar c \not\in\gamma^{(x-1)}/\rho^{(x)}$, then
$\bar c$ lies below a cell in $\gamma^{(x)}$.

Let $y=a+\Sigma^{x-1}\alpha$ denote the letter of $\mathcal A_{\alpha,x}$ 
in cell $c$ of $T$.
Then $c$ is a removable $i$-corner of the core $\tau^{(a)}$
and Definition~\ref{setvalued} implies there must be a $y$
in all $\tau^{(a)}$-removable $i$-corners. Since $\bar c$ contains 
no letter of $\mathcal A_{\alpha,x}$, it is not $\tau^{(a)}$-removable.
Therefore there is a letter $x\in\mathcal A_{\alpha,x}$
for some $x\leq y$ above or right-adj to $\bar c$.  In fact, $x<y$
since $y$ occupies only cells of residue $i$.
Assuming the later case, $\bar c$ must lie weakly lower than
$c$ since Property~\ref{thecoreprop} implies all $i$-extremals 
in $\tau^{(a)}$ that are higher than $c$ must lie at 
the end of their row (and all removable corners of $\gamma^{(x-1)}$ 
are extremal in $\gamma^{(x)}$ by horizontality).  If we choose $c$ 
to be the cell containing the lowest $y$,
there is a letter of $\mathcal A_{\alpha,x}$ that is
smaller and weakly lower than this $y$
contradicting that the lowest reading word of an affine s-v tableau is
increasing.
\end{proof}

\section{Affine Grothendieck polynomials}
\label{AGP}

Affine stable Grothendieck polynomials were introduced 
in \cite{[Lam]} in terms of the nil Hecke algebra.
Recall that the nil Hecke algebra $K$ for the type-$A$ affine 
Weyl group is generated over $\mathbb Z$ by $A_0,A_1,\ldots, A_{k}$
and relations
$$ A_i^2 = -A_i\;\text{ for all } i,
\quad
A_iA_j = A_jA_i\; \text{if}\; |i-j|> 2,
\quad
A_iA_{i+1}A_i = A_{i+1}A_iA_{i+1}
$$
where the indices are taken modulo $k+1$
\cite{[KK]}.
The algebra K is a free Z-module with basis
$\{A_w: w\in\tilde S_{k+1}\}$ where
$A_w = A_{i_1}\cdots A_{i_\ell}$ for any reduced word
${i_1}\cdots {i_\ell}$ of $w$.
In this basis, the multiplication is
given by
$$
A_iA_u =
\begin{cases}
A_{s_iu} & \text{ if }  \ell(s_iu) > \ell(u) \\
-A_u & \text{ if } \ell(s_iu) < \ell(u)
\end{cases}
$$

\smallskip

The definition of affine Grothendieck polynomials requires
elements defined by cyclically decreasing permutations.
To be precise, let $i_1\cdots i_\ell$ be a sequence of numbers
where each $i_r\in [0, k]$.
$i_1\cdots i_\ell$ is {\it cyclically decreasing} if 
no number is repeated and $j$ precedes $j-1$ (taken modulo $k+1$)
when both $j,j-1\in \{i_1,\cdots, i_\ell\}$.
If $i_1\cdots i_\ell$ is cyclically decreasing then we say
the permutation
$w = s_{i_1}\cdots s_{i_\ell}$ is cyclically decreasing.
Note that $w$ is reduced and
depends only on the set $\{i_1,\ldots,i_\ell\}$
of indices involved.
This given, consider
$$
h_i =\sum_{w\in \tilde S_{k+1}: \ell(w)=i\atop w~cyclically~decreasing} A_w
\,.
$$
Then, for any $w\in\tilde S_{k+1}$, the 
{\it affine stable Grothendieck polynomial} is defined by
\begin{equation}
\label{thomasdef}
G^{(k)}_w(x_1,x_2,\ldots) = 
\sum_{\alpha}
\langle h_{\alpha_{\ell}} h_{\alpha_{\ell-1}}
\cdots
h_{\alpha_1},A_w
\rangle
\,
x^\alpha
\,.
\end{equation}

\smallskip
We can explicitly describe the coefficients in this expression 
using certain factorizations of permutations.
Define an {\it $\alpha$-factorization} of $w$ 
to be a decomposition of the form 
$w=w^{\ell(\alpha)}\cdots w^1$ where $w^i$ is a cyclically decreasing 
permutation of length $\alpha_i$.  
From this viewpoint, the coefficient of $A_w$ in
$$
h_{\alpha_\ell}\cdots h_{\alpha_1}
= \sum_{\ell(w^\ell)=\alpha_\ell\atop
w^\ell~cyclically~dec}
A_{w^\ell}
\,
\cdots
\sum_{\ell(w^1)=\alpha_1\atop
w^1~cyclically~dec}
A_{w^1}
$$
is the signed enumeration of $\alpha$-factorizations of $w$.
Therefore,
\begin{equation}
\label{nilcoeffs}
G^{(k)}_w(x_1,x_2,\ldots) = 
\sum_{\alpha}
\,
(-1)^{|\alpha|-\ell(w)} 
\sum_{u=\alpha-factorization}
x^\alpha
\,.
\end{equation}

\smallskip

We give a bijection between affine s-v tableaux of shape 
$\lambda$ and weight $\alpha$ and {\it $\alpha$-factorizations} of 
$w_\lambda$.  From this, the stable affine Grothendieck 
polynomials indexed by grassmannian permutations are none other than 
generating functions for affine s-v tableaux.
One advantage of such an identification is that properties of
the tableaux given in prior sections immediately reveal basic 
facts about affine Grothendieck polynomials.

\smallskip

\begin{lemma}
\label{strip=cyclic}
Given a cyclically decreasing word $i_r\cdots i_1$,
let $T=\s_{i_r,1}\cdots\s_{i_1,1}(\beta)$ for any
core $\beta$. 
If the ones in $T$ occupy $r$ distinct residues,
then $(\gamma/\beta,\rho)$ is an affine s-v strip
for $\gamma=\shape(T)$ and $\rho$ the
shape obtained by deleting all ones from $T$.
\end{lemma}

\begin{proof}
Since $\s_{i_x-1,1}$ is never applied after $\s_{i_x,1}$ 
by the definition of cyclically decreasing,
$\gamma/\rho$ is horizontal.  
Let $\beta^{(0)}=\beta$ and set
$\beta^{(x)}=\s_{i_x}(\beta^{(x-1)})$ for $x=1,\ldots,r$.
An element of $\beta/\rho$ arises only when
$\s_{i_{x+1}}$ is applied to $\beta^{(x)}$
and there is $\beta^{(x)}$-removable $i_{x+1}$-corner.
In this case, all $i_{x+1}$-corners of $\beta^{(x)}$
are $\beta$-removable
since $\beta^{(x)}/\beta$ has no $i_{x+1}$-residue.
Further, any cell $c'$ right-adj to 
a $\beta$-removable $i_{x+1}$-corner $c$
has residue $i_{x+1}+1$ and thus is not
in $\beta^{(x)}$ by definition of cyclically decreasing.
Therefore $c$ is either $\beta^{(x)}$-removable 
(and in $\beta/\rho$) or it is $\gamma$-blocked.

It thus remains to prove that $\gamma/\beta$ is an 
affine $r-m$ strip, where $m=|Res(\beta/\rho)|$.
Since the ones in $T$ occupy $r$ distinct residues,
$Res(\gamma/\rho)=\{i_1,\ldots,i_r\}$ and
$\beta^{(x-1)}$ has a removable or an addable $i_x$-corner.
Since $\{i_1,\ldots,i_r\}$ are distinct
and a core never has both an addable and removable
corner of the same residue, $Res(\beta^{(x)}/\beta)
\cap Res(\beta/\rho)=\emptyset$.
Therefore, by Remark~\ref{addonezerosi}(1),
$|\kbnd(\gamma)|=|\kbnd(\beta)|+r-m$.
\end{proof}

\smallskip

\begin{theorem}
\label{alphabij}
For $\lambda\in\mathcal P^k$,
there is a bijection between
$\mathcal T^{k}_\alpha(\lambda)$ and
the set of $\alpha$-factorizations for $w_\lambda$.  
\end{theorem}
\begin{proof}
($\Leftarrow$)
Consider an $\alpha$-factorization $w_\lambda=w^\ell w^{\ell-1}\cdots w^1$
and let $i_{\Sigma^x\alpha}\cdots i_{\Sigma^{x-1}\alpha+1}$
be a cyclically decreasing word for $w^{x}$,
for each $x=1,\ldots,\ell=\ell(\alpha)$.
From this, iteratively
construct $U^{(x)} = \mathfrak s_{i_{\Sigma^x\alpha},x}\cdots 
\mathfrak s_{i_{\Sigma^{x-1}\alpha+1},x}
(U^{(x-1)})$.  Let $\gamma^{(x)} = \shape(U^{(x)})$ and 
let $\rho^{(x)}$ be
the shape of $U^{(x)}$ minus its cells containing the letter $x$.
By Theorem~\ref{setvalued2}, it suffices to show that
$(\gamma^{(x)}/\gamma^{(x-1)},\rho^{(x)})$ is an
affine s-v $\alpha_x$-strip.

Since $w_\lambda$ is grassmannian and $w^1$ is
cyclically decreasing, $w^1=s_{\alpha_1-1}\cdots s_0$.  
Thus the result holds for $x=1$ since
$\gamma^{(1)}=\s_{\alpha_1-1}\cdots \s_0\emptyset$
is horizontal and $\rho^{(1)}=\emptyset$.  
By induction,
$\gamma^{(\ell-1)}$ is a core and by Remark~\ref{hasaddableorremovable},
the letter $\ell$ in $U^{(\ell)}$ occupies $\alpha_\ell$ distinct 
residues.  The result then follows by applying
Lemma~\ref{strip=cyclic} since
$\gamma^{(\ell)}= \s_{i_{\Sigma^\ell\alpha}}\cdots 
\s_{i_{\Sigma^{\ell-1}\alpha+1}}\gamma^{(\ell-1)}$.

\smallskip

\noindent ($\Rightarrow$):
Given $T\in \mathcal T^k_{\alpha}(\lambda)$,  let
$w^x= s_{j_{\Sigma^x\alpha}}\cdots s_{j_{\Sigma^{x-1}\alpha+1}}$ 
where $j_a$ denotes the 
residue of letter $a$ in $T$, for $x=1,\ldots,\ell(\alpha)$.  
Proposition~\ref{bij} implies that $w_\lambda = w^\ell\cdots w^1$
and it remains to show that $w^x$ is cyclically decreasing.

The letters of $\mathcal A_{\alpha,x}$ occupy residues
$\mathcal S= \{j_{\Sigma^{x-1}\alpha+1},\ldots,j_{\Sigma^x\alpha}\}$.
The definition of affine s-v tableau implies residues in $\mathcal S$ 
are distinct, the lowest reading word of $\mathcal A_{\alpha,x}$ is 
increasing, and $\gamma/\rho$ is horizontal, where 
$\gamma=\shape(T_{\leq \Sigma^x\alpha})$
and $\rho$ is the shape of $T_{\leq \Sigma^x\alpha}$ minus
cells containing an element of $\mathcal A_{\alpha,x}$.
Suppose $i,i-1\in\mathcal S$ and $i-1$ precedes $i$ in 
$j_{\Sigma^x\alpha}\cdots j_{\Sigma^{x-1}\alpha+1}$.
Then there are letters $t_1(i)$ and $t_2(i-1)$ in
$\mathcal A_{\alpha,x}$ where $t_1<t_2$.
Since the lowest reading word is increasing, the lowest
$t_2$ occurs to the right of $t_1$ and
they do not lie in the same row since $\alpha_x\leq k$.
Further, there is no element of $\mathcal A_{\alpha,x}$ 
right-adj to $t_2(i-1)$ since this could only be $t_1(i)$.
Therefore, all extremals of residue $i-1$ lie at the
end of their row by Property~\ref{thecoreprop}.  However, $t_1(i)$ is right-adj
to an extremal of residue $i-1$ by the horizontality
of $\gamma/\rho$.
\end{proof}

\begin{corollary}
\label{thomas=me}
For any $\lambda\in\mathcal P^k$, 
\begin{equation}
\label{grotgen}
G_{\lambda}^{(k)} =  \sum_{T\in \mathcal T^k(\lambda)}
(-1)^{|\lambda|+|w(T)|}\,
x^{w(T)}\,,
\end{equation}
where $G^{(k)}_\lambda = G^{(k)}_{w_\lambda}$. 
\end{corollary}

This interpretation for the  $G_\lambda^{(k)}$
allows us to establish a number of properties using our
results on affine s-v tableaux.  To start,
Mark Shimozono conjectured that affine Grothendieck polynomials 
for the grassmannian reduce to Grothendieck polynomials in limiting
cases of $k$.  In fact, we find precisely that

\begin{property}
\label{Gkinfty}
If $h(\lambda)\leq k$ then $G^{(k)}_\lambda = G_\lambda$.
\end{property}
\begin{proof}
Proposition~\ref{propsvkt1} tells us that the elements of
$\mathcal T^k(\lambda)$ are set-valued tableaux of 
shape $\lambda$ when $h(\lambda)\leq k$.
The result thus follows from
Corollary~\ref{thomas=me} and the definition 
for Grothendieck polynomials.
\end{proof}

It was shown in \cite{[Lam]} that $G^{(k)}_w$ of \eqref{thomasdef} 
are symmetric functions.  Thus, letting
$$
\mathcal K_{\lambda\alpha}^{(k)} = 
\left|\mathcal T^k_\alpha(\lambda)\right|\,
$$
enumerate the affine s-v tableaux,
we deduce a symmetry of this affine K-theoretic 
refinement of the Kostka numbers from Corollary~\ref{thomas=me}.
\begin{corollary}
\footnote{
A direct combinatorial proof of this
symmetry will appear in \cite{[BM]} using an involution 
from the set of $\alpha$-factorizations of $w_\lambda$
to the set of $\hat\alpha$-factorizations,
where $\hat\alpha$ is obtained by transposing two adjacent components 
of $\alpha$.   The involution generalizes the 
Lascoux-Sch\"utzenberger symmetric group action on words.}
\label{corosym}
Given any $\lambda\in\mathcal P^k$ and
$k$-bounded composition $\alpha$,
$$
\mathcal K_{\lambda\alpha}^{(k)}=
\mathcal K_{\lambda\beta}^{(k)}\,\quad
$$
for any rearrangement $\beta$ of $\alpha$.
\end{corollary}

The affine Grothendieck polynomial can then be written as,
for $\lambda\in\mathcal P^k$, 
\begin{equation}
\label{grotinm}
G_\lambda^{(k)} = \sum_{\mu\in\mathcal P^k} 
(-1)^{|\lambda|+|\mu|}\,
\mathcal K^{(k)}_{\lambda\mu} \, m_\mu\,.
\end{equation}

Our earlier result showing affine s-v tableaux are simply
$k$-tableaux in certain cases also enables us to refine
this expansion and connect affine Grothendieck polynomials
to dual $k$-Schur functions.

\begin{property}
\label{trikostka}
For any $k$-bounded partitions $\lambda$ and $\mu$,
\begin{equation}
\mathcal K_{\lambda\mu}^{(k)}=
\begin{cases}
1 \quad \text{when}\quad \mu=\lambda\\
0 \quad\text{when}\quad |\mu|=|\lambda| \;\;\text{and}\;\;
\lambda \ntrianglerighteq \mu
\\
0 \quad\text{when}\quad |\mu|<|\lambda|
\end{cases}
\end{equation}
\end{property}
\begin{proof}
Consider an affine s-v tableau $T$ of shape $\core(\lambda)$ and
weight $\mu$.  We have that $|\lambda|\geq |\mu|$ by
Corollary~\ref{tridegree}.  Further, Proposition~\ref{propsvkt2} implies 
that $T$ is a $k$-tableau when $|\mu|=|\lambda|$,
in which case the desired relation follows from \eqref{unikkostka}.
\end{proof}

We thus have the unitriangularity relation:
\begin{equation}
\label{grotinmtri}
G_\lambda^{(k)} =
m_\lambda + 
\sum_{\mu\in\mathcal P^k\atop\mu\lhd\lambda}
\mathcal K_{\lambda\mu}^{(k)}\, m_\mu
\,+\,
\sum_{\mu\in\mathcal P^k\atop |\mu|>|\lambda|}
(-1)^{|\lambda|+|\mu|}\,
\,\mathcal K_{\lambda\mu}^{(k)}\, m_\mu
\,,
\end{equation}
and it follows immediately that these
affine Grothendieck polynomials are a basis.

\begin{property}
\label{grotbas}
$\{G_\lambda^{(k)}\}_{\lambda\in\mathcal P^k}$ is a basis for 
$\Lambda/\mathcal I^{k}$.
\end{property}

In analogy to \eqref{grothins}, we also deduce
that the dual $k$-Schur expansion of an affine
Grothendieck polynomial has integer coefficients
and the polynomial made up of the lowest homogeneous degree 
terms is precisely a dual $k$-Schur function.

\begin{property}
\label{grotindualks}
For any $k$-bounded partition $\lambda$,
\begin{equation}
G_\lambda^{(k)} \quad =\quad 
\mathfrak S_\lambda^{(k)} + \sum_{\mu\in\mathcal P^k\atop|\mu|>|\lambda|} 
a_{\lambda\mu}^k\,\mathfrak S_\mu^{(k)}
\qquad\text{for}\quad a_{\lambda\mu}^k\in\mathbb Z\,.
\end{equation}
\end{property}
\begin{proof}
The bottom degree terms of expression \eqref{grotinmtri} 
matches the monomial expansion \eqref{dualksenm} for
the dual $k$-Schur function $\mathfrak S_\lambda^{(k)}$
since $\mathcal K_{\lambda\mu}^{(k)}=K_{\lambda\mu}^{(k)}$
when $|\mu|=|\lambda|$ by Proposition~\ref{propsvkt2}. 
The higher degree terms involve $m_\mu\in
\Lambda/\mathcal I^k$ and can thus
be expanded into the $\{\mathfrak S_\lambda^{(k)}\}$-basis.
The coefficient remain integral by the unitriangularity
of expansion \eqref{dualksenm}.
\end{proof}

On one hand, as their name suggests, the affine Grothendieck
polynomials can be viewed as an affine analog of the Grothendieck 
polynomials.   At a fundamental level, because the expansion 
coefficients in \eqref{grothins} are in fact positive (up to a
degree-alternating sign), Thomas Lam conjectured the
same about the coefficients $a_{\lambda\mu}^k$.  In the same
vein, it was proven in \cite{[Lenart]} that  the coefficients in
\begin{equation}
s_\lambda = \sum_\mu f_{\lambda\mu}\, G_\mu
\end{equation}
have a simple combinatorial interpretation as
the number of certain restricted skew tableaux.
Evidence suggests that the affine analog of this
identity is also positive.
\begin{conjecture}
For any $k$-bounded partition $\lambda$,
the coefficients $f_{\lambda\mu}^{k}$ in
\begin{equation}
\mathfrak S^{(k)}_\lambda = \sum_{\mu\in\mathcal P^k\atop |\mu|\geq|\lambda|}
f_{\lambda\mu}^{k}\, G_\mu^{(k)}  
\end{equation}
are non-negative integers.  
\end{conjecture}
On the other hand, instead viewing affine Grothendieck polynomials as the
$K$-theoretic analog of dual $k$-Schur functions suggests
that these polynomials satisfy even more refined combinatorial properties.
For example, it is proven in the forthcoming paper 
\cite{[LLMS2]} that the coefficients in 
\begin{equation}
\mathfrak S_\lambda^{(k+1)} =
\sum_{\mu\in\mathcal P^k}
a_{\lambda,\mu}^{k+1,k}\, \mathfrak S_\mu^{(k)}  
{\mod \mathcal I^k} 
\end{equation}
are non-negative integers.  Since a dual $k$-Schur function reduces 
to a Schur function for large $k$, this expression can be
iterated to imply the positivity of coefficients in
\begin{equation}
\label{sindual}
s_\lambda = \sum_{\mu\in\mathcal P^k}
a_{\lambda,\mu}^k \,\mathfrak S_\mu^{(k)} 
\mod \mathcal I^k\,,
\end{equation}
for any $k>0$.
Naturally following suite in our setting leads to
the $K$-theoretic version of these ideas.

\begin{conjecture}
For any $k+1$-bounded partition $\lambda$,
the coefficients $d_{\lambda\mu}^{k+1,k}$ in
\begin{equation}
\label{GinG}
G_\lambda^{(k+1)} = \sum_{\mu\in\mathcal P^k\atop |\mu|\geq|\lambda|}
(-1)^{|\lambda|+|\mu|}\,
d_{\lambda\mu}^{k+1,k}\, G_\mu^{(k)}  
\mod \mathcal I^k
\end{equation}
are non-negative integers.  By Property~\ref{Gkinfty} this 
implies
\begin{equation}
\label{sinG}
G_\lambda = \sum_{\mu\in\mathcal P^k\atop |\lambda|\leq |\mu|} 
(-1)^{|\lambda|+|\mu|}\,d_{\lambda\mu}^k\,
G^{(k)}_\mu\,\quad\text{where}\quad d_{\lambda\mu}^k\in\mathbb N\,.
\end{equation}
\end{conjecture}

Note that \eqref{sinG} is the analog of \eqref{sindual} 
where the Schur function $s_\lambda$ is considered to 
be $\mathfrak s^{(\infty)}_\lambda$ and the dual $k$-Schur 
functions are then all replaced
by their $K$-theoretic counterparts, the affine
Grothendieck polynomials.
Alternatively, if we do not interpret $s_\lambda$ as
a dual $\infty$-Schur auction, we can derive a
conjecture about the affine Grothendieck  expansion of 
a Schur function.

\begin{conjecture}
For any $k$-bounded partition $\lambda$,
the coefficients $d_{\lambda\mu}^{k}$ in
\begin{equation}
s_\lambda = \sum_{\mu\in\mathcal P^k\atop |\mu|\geq|\lambda|}
d_{\lambda\mu}^{k}\, G_\mu^{(k)}  
\mod \mathcal I^k
\end{equation}
are non-negative integers.  
\end{conjecture}

\section{$k$-$K$-Schur functions}

Thomas Lam conjectured in an FRG wiki post, and at the 2008
FRG problem solving session, that there is a basis 
$g_\lambda^{(k)}$ of $\Lambda^{(k)}$ such that
\begin{enumerate}
\item 
$\langle g_\lambda^{(k)},G_\mu^{(k)}\rangle = \delta_{\lambda\mu}$
\item as $k\to\infty$, $g_\lambda^{(k)}$ reduces to the dual
Grothendieck polynomial $g_\lambda$
\item The top homogeneous component of $g_\lambda^{(k)}$ is
the k-Schur function $s_\lambda^{(k)}$
\item $g_\lambda^{(k)}$ can be expanded positively in terms
of k-Schur functions.
\end{enumerate}

\smallskip

It was with this in mind that we began the study of a second
family of polynomials called {\it $k$-$K$-Schur functions}.  
Our point of departure is similar to what was done to define 
dual Grothendieck \eqref{defdualg} and $k$-Schur functions 
\eqref{kschurdef}, but now exploiting the invertibility of the matrix 
$||\mathcal K^{(k)}||_{\lambda,\mu\in \mathcal P^k}$
given by Property~\ref{trikostka}.

\begin{definition}
\label{kgrotdef}
{\it $k$-$K$-Schur functions} $g_\lambda^{(k)}$
are defined by the system of equations,
\begin{equation}
\label{e1}
h_\lambda = \sum_{\mu\in\mathcal P^k\atop
|\mu|\leq|\lambda|} 
(-1)^{|\lambda|+|\mu|}\,
\mathcal K_{\mu\lambda}^{(k)}
\, g_\mu^{(k)}\,\quad\text{
for all $\lambda\in\mathcal P^k$.}
\end{equation}
\end{definition}

In particular, we let $||\bar{\mathcal K}^{(k)}||$ denote the inverse
of $||{\mathcal K}^{(k)}||$ and invert \eqref{e1}.
The conditions on this matrix imposed by Property~\ref{trikostka}
imply
\begin{equation}
\label{invertkKs}
g_\lambda^{(k)}=
h_\lambda\, + 
\,
\sum_{\mu\in\mathcal P^k\atop
\mu\rhd\lambda}
\bar{\mathcal K}_{\mu\lambda}^{(k)}\,
h_\mu +
\sum_{\mu\in\mathcal P^k\atop
|\mu|<|\lambda|}
(-1)^{|\mu|+|\lambda|}
\bar{\mathcal K}_{\mu\lambda}^{(k)}\,
h_\mu 
\,.
\end{equation}
From this, we extract a number of properties 
including proofs of Conjectures (1--3). 

\begin{property}
The set $\left\{g_\lambda^{(k)}\right\}_{\lambda_1\leq k}$ forms a 
basis for $\Lambda^{(k)}$.
\end{property}

\begin{property}
For all $\lambda,\mu\in\mathcal P^k$,
$$
\langle g_\lambda^{(k)},G_\mu^{(k)}\rangle = \delta_{\lambda\mu}\,.
$$
\end{property}
\begin{proof}
From Equation~\eqref{invertkKs} for $k$-$K$-Schur functions, we have
that
$$
\langle G_\mu^{(k)},g_\lambda^{(k)}\rangle = 
\langle 
\sum_{|\mu|\leq |\beta|} (-1)^{|\mu|+|\beta|}\,
\mathcal K^{(k)}_{\mu\beta}\, m_\beta 
\,,
\sum_{|\alpha|\leq |\lambda|} (-1)^{|\lambda|+|\alpha|}
\, \bar{\mathcal K}^{(k)}_{\alpha\lambda}\, h_\alpha
\rangle
\,,
$$
implying 
$$
\langle G_\mu^{(k)},g_\lambda^{(k)}\rangle = 
\,
(-1)^{|\lambda|+|\mu|}
\sum_\beta
{\mathcal K}^{(k)}_{\mu\beta} 
\bar{\mathcal K}^{(k)}_{\beta\lambda} 
\,=\,\delta_{\lambda\mu}
$$
by the duality of $\{m_\beta\}$ and $\{h_\alpha\}$.
\end{proof}

We have seen that the term of lowest degree in 
the affine Grothendieck polynomial is a dual $k$-Schur
function.  The affine analog of \eqref{dualgins} is that
the highest degree 
term of a $k$-$K$-Schur function is a $k$-Schur function.
We also find that their expansion coefficients in terms
of $k$-Schur functions and the dual Grothendieck polynomials
are integers.

\begin{property}
\label{grotkschur}
For any $k$-bounded partition $\lambda$,
\begin{equation}
\label{triom}
g_{\lambda}^{(k)}=s_{\lambda}^{(k)}+ \sum_{\mu\in\mathcal P^k\atop
|\mu|<|\lambda|} d_{\lambda\mu}^{(k)}\, 
s^{(k)}_\mu
 \;\;\; for\;\; d_{\lambda\mu}^{(k)}\in\mathbb Z\,.
\end{equation}
\end{property}
\begin{proof}
Proposition~\ref{propsvkt2} implies that
$\mathcal K_{\mu\lambda}^{(k)}=K_{\mu\lambda}^{(k)}$
for all $|\lambda|=|\mu|$.  Making this replacement in
Equation~\eqref{invertkKs} gives that
$$
g_\lambda^{(k)}=
s_\lambda^{(k)} +
\sum_{\mu\in\mathcal P^k\atop
|\mu|<|\lambda|}
(-1)^{|\mu|+|\lambda|}
\bar{\mathcal K}_{\mu\lambda}^{(k)}\,
h_\mu 
\,.
$$
We can then expand $h_\mu$ in terms
of $k$-Schur functions using \eqref{kschurdef}
and since $\bar{\mathcal K}_{\mu\lambda}^{(k)}\in\mathbb Z$
by unitriangularity, we find
that the expansion coefficients of \eqref{triom} are
integers. 
\end{proof}

The long-standing conjecture that $k$-Schur functions are
Schur positive,
\begin{equation}
s_\lambda^{(k)} = \sum_{\mu}
b_{\lambda,\mu}^k\; s_\mu  
\quad\text{where}\quad b_{\lambda\mu}^{k}\in\mathbb N\,,
\end{equation}
has recently been proven in \cite{AB} and in \cite{[LLMS2]}.
The path to proving this result that is taken in \cite{[LLMS2]}
proves the more refined property
that
\begin{equation}
s_\lambda^{(k)} = \sum_{\mu\in\mathcal P^{k+1}}
b_{\lambda,\mu}^{k,k+1}\; s_\mu^{(k+1)}  
\quad\text{where}\quad b_{\lambda\mu}^{k,k+1}\in\mathbb N\,.
\end{equation}
The Schur positivity follows from this because a $k$-Schur 
function reduces to a Schur function for large $k$.

\smallskip

We conjecture that the theory of $k$-$K$-Schur functions follows
a similar path. To be precise,
the homogeneous symmetric functions that arise in \eqref{invertkKs} 
can be integrally expanded in terms of the dual Grothendieck polynomials 
by \eqref{defdualg}.
\begin{property}
For any $k$-bounded partition $\lambda$,
\begin{equation}
\label{grotpositive}
g_{\lambda}^{(k)}=g_{\lambda}+ 
\sum_{\mu\rhd\lambda}
b_{\lambda\mu}\, g_\mu \,+\,
\sum_{|\mu|<|\lambda|} 
b_{\lambda\mu}\, g_\mu \;\;\; 
for\;\; b_{\lambda\mu}\in\mathbb Z\,.
\end{equation}
\end{property}

Similarly, the homogeneous symmetric functions that arise 
in  \eqref{invertkKs} can be instead be expanded integrally
in terms of the $k+1$-$K$-Schur functions using 
Definition~\ref{kgrotdef}.  

\begin{property}
\label{trigro}
For any $\lambda\in\mathcal P^k$,
\begin{equation}
g_{\lambda}^{(k)}=g_{\lambda}^{(k+1)}+ 
\sum_{\mu\rhd\lambda}
b_{\lambda\mu}^k\, g_\mu^{(k+1)} \,+\,
\sum_{|\mu|<|\lambda|} 
b_{\lambda\mu}^k\, g_\mu^{(k+1)} \;\;\; 
for\;\; b_{\lambda\mu}^k\in\mathbb Z\,.
\end{equation}
\end{property}

\begin{conjecture}
For all $\lambda,\mu\in\mathcal P^k$,
the integer coefficient $(-1)^{|\lambda|+|\mu|}\,b_{\lambda\mu}^k$ 
is non-negative.
\end{conjecture}

Then, when $k$ is large, a $k$-$K$-Schur function reduces
simply to a dual Grothendieck polynomial.

\begin{property}
\label{kinfty}
If $|\lambda|\leq k$, then $g_\lambda^{(k)}=g_\lambda$.
\end{property}
\begin{proof}
Proposition~\ref{propsvkt1} implies that
$\mathcal K_{\lambda\mu}^{(k)}=\mathcal K_{\lambda\mu}$ 
when $h(\lambda)\leq k$.
In particular, if $|\lambda|\leq k$ then
all partitions $\mu$ where $|\mu|\leq |\lambda|$
have $k$-bounded  hook-length.
\end{proof}

From this, iterating Property~\ref{trigro} will eventually lead 
to the positive (up to alternating sign) expansion
coefficients in \eqref{grotpositive}
of a $k$-$K$-Schur function in terms
of dual Grothendieck polynomials.

\begin{conjecture}
For all $\lambda,\mu\in\mathcal P^k$,
the integer coefficient $(-1)^{|\lambda|+|\mu|}\,b_{\lambda\mu}$ 
is non-negative.
\end{conjecture}

\begin{property}
\label{kss}
For any partition $\lambda$ with $h(\lambda)\leq k$,
we have that 
\begin{equation}
\label{kinftylower}
g_\lambda^{(k)}=g_\lambda + \text{lower degree terms}\,.
\end{equation}
\end{property}
Note: we conjecture that all the lower degree terms cancel.

\begin{proof}
When $k\geq h(\lambda)$,
it was shown in \cite{[LMproofs]} that
$s_\lambda^{(k)} = s_{\lambda}$.
Thus, the $k$-Schur expansion of Property~\ref{grotkschur} reduces
in this case to
\begin{equation}
g_\lambda^{(k)}=s_\lambda + \text{lower degree terms}\,.
\end{equation}
The result the follows from \eqref{dualgins}
expressing $g_\lambda$ as $s_\lambda$ plus lower degree terms.
\end{proof}

\section{Pieri rules}

In addition to the basic properties of $k$-$K$-Schur functions
extracted from the definition, we have also 
determined explicit Pieri rules for these polynomials.

\subsection{Row Pieri Rule}

\begin{theorem}
\label{kpieri}
For any $k$-bounded partition $\lambda$ and $r\leq k$,
\begin{equation}
\label{kpierieq}
g^{(k)}_r\,g_\lambda^{(k)} = 
\sum_{(\mu,\rho)\in \mathcal H_{\lambda,r}^{k}}
(-1)^{|\lambda|+r-|\mu|}
\,g_\mu^{(k)}
\,,
\end{equation}
where
$
\mathcal H^{k}_{\lambda,r} = 
\{ (\mu,\rho): (\core(\mu)/\core(\lambda),\rho)
=\text{affine~set-valued~$r$-strip}
\}
\,.
$
\end{theorem}

\begin{example}
$$
g_2^{(3)}\,g_{3,2,1}^{(3)} 
=
g^{(3)}_{3, 2, 2, 1}+
g^{(3)}_{3, 3, 1, 1}-
g^{(3)}_{3, 2, 1, 1}-
2g^{(3)}_{3, 2, 2}+
g^{(3)}_{3, 2, 1}
$$
\end{example}

\begin{proof}
Note that $g_\ell^{(k)}=h_\ell$.
Since the $k$-$K$-Schur functions form a basis of $\Lambda^{(k)}$, 
there is an expansion
\begin{equation}
h_\ell\,g_\nu^{(k)}
= \sum_{\mu} c_{\mu\nu}\, g_\mu^{(k)}
\,,
\end{equation}
for some coefficients $c_{\mu\nu}$.  To determine the
$c_{\mu\nu}$, we examine
$h_\ell h_\lambda$.  Using the $k$-$K$-Schur expansion \eqref{e1} 
for $h_\lambda$, we find that
\begin{equation}
\label{pierieq2}
h_{\ell} h_\lambda=
\sum_{\nu } \mathcal K^{(k)}_{\nu\lambda} \,h_\ell\,g_\nu^{(k)}
=
\sum_{\nu } \mathcal K^{(k)}_{\nu\lambda} 
\sum_{\mu} c_{\mu\nu}\, g_\mu^{(k)}
\,.
\end{equation}
On the other hand, we can use \eqref{e1} to expand $h_{\ell} h_\lambda=
h_{\tau}$, where $\tau$ is the partition rearrangement of
$(\ell,\lambda)$:  
\begin{equation}
\label{pierieq1}
h_\ell h_\lambda = 
h_{\tau} = 
\sum_{\mu} \mathcal K^{(k)}_{\mu\tau} \,g_\mu^{(k)} 
= \sum_{\mu} \mathcal K^{(k)}_{\mu\,(\lambda,\ell)} \,g_\mu^{(k)} 
\,,
\end{equation}
where the last equality holds by Corollary~\ref{corosym}.
Then, since Theorem~\ref{setvalued2} implies 
\begin{equation}
\label{kids}
\mathcal K_{\mu\, {(\lambda,\ell)}}^{(k)} \,=\,
\sum_{\nu: \mu\in \mathcal H^{k}_{\nu,\ell}}
\mathcal K_{\nu\lambda}^{(k)}
\,,
\end{equation}
we have
\begin{equation}
\label{pierieq1}
h_\ell h_\lambda=
\sum_{\mu } 
\sum_{\nu:\mu\in \mathcal H^{(k)}_{\nu,\ell}}
\mathcal K^{(k)}_{\nu\lambda} \,g_\mu^{(k)}
\,.
\end{equation}
Equating the coefficient of $g_\mu^{(k)}$ in the right 
side of this expression to that of \eqref{pierieq2} to get the system:
\begin{equation}
\sum_{\nu:\mu\in \mathcal H^{(k)}_{\nu,\ell}}
\mathcal K^{(k)}_{\nu\lambda} 
= \sum_{\nu} \mathcal K^{(k)}_{\nu\lambda} c_{\mu\nu}
\,.
\end{equation}
We thus find our desired solution
\begin{equation}
c_{\mu \nu} = 
\left\{ 
\begin{array}{ll} 
1 & {\rm if~} \mu \in \mathcal H_{\nu,\ell}^{(k)} \\
0 & {\rm otherwise}
\end{array}
\right\}
\,.
\end{equation}
It is unique since any other solution satisfies
\begin{equation}
\sum_{\nu} \mathcal K^{(k)}_{\nu\lambda}\left(
\tilde c_{\mu\nu} - c_{\mu\nu}
\right) = 0\,,
\end{equation}
and the invertibility of $\mathcal K^{(k)}_{\nu\lambda}$
implies $\tilde c_{\mu\nu}=c_{\mu\nu}$.
\end{proof}

The Pieri rule can equivalently be phrased in the notation 
of affine permutations.
In particular, Theorem~\ref{alphabij} identifies 
affine s-v strips with cyclically decreasing words
and we know (e.g. Remark~\ref{part2word}) that
that $w_{(r)}= s_{r-1}\cdots s_{0}$.

\begin{corollary}
For any $w\in\tilde S^0_{k+1}$ and $r\leq k$,
\begin{equation}
\label{kpiericor}
g^{(k)}_{s_{r-1}\cdots s_{0}}\, g^{(k)}_{w} = 
\sum_{v=uw: \ell(u)=r
\atop {
u~cyclically~decreasing
} 
}
(-1)^{\ell(w)+r-\ell(v)}\,
 g^{(k)}_{v}
\,.
\end{equation}
\end{corollary}

Note that
the highest degree terms in the rhs of \eqref{kpierieq}
are simply the terms given by the Pieri rule 
\cite{[LMproofs]} for $k$-Schur functions:
$$
h_\ell\, g_\nu^{(k)} = h_\ell \,s_\nu^{(k)} + lower~terms
\,,
$$  
obtained by adding affine $\ell$-strips to $\core(\nu)$.

\subsection{Column Pieri Rule}

There is also a combinatorial rule to compute the $k$-$K$-Schur 
function expansion of $g_{1^\ell}^{(k)} g_\nu^{(k)}$ in 
terms of vertical strips rather than horizontal.  The 
dual Grothendieck polynomial indexed by a column is
\begin{equation}
\label{g1}
g_{1^\ell}=\sum_{j=1}^{\ell}\binom{\ell-1}{j-1} e_{j}
\,,
\end{equation}
and $g_{1^\ell}^{(k)}= g_{1^\ell}$ when $\ell \leq k$
by Property~\ref{kinfty}.
To determine the associated Pieri rule, we start with
a ``K-theoretic" version of Newton's formula
(e.g. \cite{Macbook}):
\begin{equation}
\label{newton}
\sum_{r=0}^{\ell} (-1)^{r}\, h_{\ell-r} \,e_r = 0
\,.
\end{equation}
\begin{proposition}
\label{innewton}
For any integer $\ell\geq 0$,
\begin{equation}
\sum_{r=0}^{\ell} \sum_{j=0}^{r}
(-1)^{j+r} \binom{r-2}{j} g_{\ell-r}\,
g_{1^{r-j}} = 0
\,.
\end{equation}
\end{proposition}
\begin{proof}
By expression \eqref{g1} for $g_{1^\ell}$, this follows from
the identity
\begin{equation}
\sum_{r=0}^{\ell} \sum_{j=0}^{r}
\sum_{t=1}^{r-j}
(-1)^{j+r} \binom{r-2}{j}
\binom{r-j-1}{t-1} 
h_{\ell-r}\,
e_{t}
= 0
\,.
\end{equation}
Equivalently,
\begin{equation}
\label{knewt}
\sum_{r=0}^{\ell} 
\sum_{t=1}^{r}
(-1)^r
\left(
\sum_{j=0}^{r-t}
(-1)^{j} \binom{r-2}{j}
\binom{r-j-1}{t-1} 
\right)
h_{\ell-r}\,
e_{t}
= 0
\,.
\end{equation}
In fact, the orthogonality identity implies that
\begin{equation}
\sum_{j=0}^{r-t}
(-1)^{j} \binom{r-2}{r-2-j}
\left(
\binom{r-j-2}{t-1} 
+ \binom{r-j-2}{t-2} 
\right)
\qquad
\qquad
\end{equation}
\begin{equation}
\qquad
=
\;
\;
(-1)^{r+t-1} \delta_{r-2,t-1} +
(-1)^{r-t} \delta_{r-2,t-2}
\end{equation}
and thus the l.h.s. of \eqref{knewt} reduces to
\begin{equation}
\sum_{r=0}^{\ell} 
\sum_{t=1}^{r}
(-1)^r
\left(
(-1)^{r+t-1} \delta_{r-2,t-1} +
(-1)^{r-t} \delta_{r-2,t-2}
\right)
h_{\ell-r}\,
e_{t}
\qquad \qquad
\end{equation}
\begin{equation}
\qquad \qquad
\qquad \qquad
= 
\sum_{r=0}^{\ell} 
(-1)^r
h_{\ell-r}\,
e_{r-1}
+
\sum_{r=0}^{\ell} 
(-1)^r
h_{\ell-r}\,
e_{r}
\end{equation}
which vanishes by Newton's identity.
\end{proof}

\begin{theorem}
\label{ekpieri}
For any $k$-bounded partition $\lambda$ and integer $r\leq k$, 
\begin{equation}
\label{epieri}
g_{1^r}^{(k)}\,g_\lambda^{(k)} = 
\sum_{(\mu,\rho)\in \mathcal E_{\lambda,r}^{k}} 
(-1)^{|\lambda|+r-|\mu|}
g_{\mu}^{(k)} \,,
\end{equation}
where $(\mu,\rho)\in\mathcal E_{\lambda,r}^{k}$ iff
$(\mu^{\omega_k},\rho') \in \mathcal H^k_{\lambda^{\omega_k},r}$.
\end{theorem}

\begin{example}
$$
g_{1,1}^{(3)}\,g_{3,2,1}^{(3)} 
=
g^{(3)}_{3, 2, 1,1, 1}+
g^{(3)}_{3, 2, 2, 1}-
g^{(3)}_{3, 2, 1, 1}-
g^{(3)}_{3, 2, 2}+
g^{(3)}_{3, 2, 1}
$$
\end{example}

\begin{proof}
Since $g_1=h_1$,
Theorem~\ref{kpieri} implies the case when $r=1$ and
we assume by induction that the action of $g_{1^{s}}$ 
for all $s<r$ is given by \eqref{epieri}.   To prove our 
assertion for multiplication by $g_{1^r}$, 
since Proposition~\ref{innewton} can be rewritten as
\begin{equation}
\sum_{s=0}^{r-1} \sum_{j=0}^{r-s}
(-1)^{s} \binom{s+j-2}{j} g_{r-s-j}\,
g_{1^{s}}
+(-1)^r g_{1^r} =0
\,,
\end{equation}
it suffices to show
\begin{equation}
\label{sumzero}
\sum_{s=0}^{r-1} \sum_{j=0}^{r-s}
(-1)^{s} \binom{s+j-2}{j} g_{r-s-j}\,
g_{1^{s}} \,g_\lambda^{(k)} 
\; +\;
\!\! \!\!
\sum_{(\mu^{\omega_k},\rho)\in \mathcal H_{\lambda^{\omega_k},r}^{(k)}} 
(-1)^{|\mu|-|\lambda|}\,
g_{\mu}^{(k)}
=0 \,.
\end{equation}
We claim that the coefficient of $g_\nu^{(k)}$ in the left 
side of this expression is zero
for any $\nu\in\mathcal P^k$. 

\smallskip

By induction, for $s<r$, the coefficient of 
$g_{\nu}^{(k)}$ in $g_{r-s-j}g_{1^{s}} 
g_\lambda^{(k)}$ is $(-1)^{|\nu|-|\lambda|-r+j}$ times 
the number of {\it vh-fillings} with weight $(s,r-s-j)$ 
defined by:

\smallskip

\begin{enumerate}
\item[(i)] 
letter $x$ lies in cells of $\core(\mu)/\rho$ where
$(\mu^{\omega_k},\rho')\in\mathcal H^k_{\lambda^{\omega_k},s}$

\smallskip

\item[(ii)]  
letter $y$ lies in cells of $\core(\nu)/\tau$ where
$(\nu,\tau)\in\mathcal H^k_{\mu,r-s-j}$
\end{enumerate}

\noindent
Denote the set of such fillings by 
$\mathcal {VH}^k_{s,r-s-j}(\nu,\lambda)$.  
This given, for fixed $\nu,\lambda\in\mathcal P^k$,
\eqref{sumzero} follows 
by proving the identity
\begin{equation}
\label{epieri2}
\sum_{s=0}^{r} \sum_{j=0}^{r-s}
(-1)^{r-s-j}
\binom{s+j-2}{j} 
|\mathcal {VH}^k_{s,r-s-j}(\nu,\lambda)|
=0 \,.
\end{equation}
We take a combinatorial approach.
First we rewrite this expression as a single sum by
introducing ordered multisets of signed $vh$-fillings to 
account for the binomial numbers and then describe a
sign-reversing involution to achieve cancellation.
The desired involution will act by permuting certain 
\dit{free} entries of $T\in \mathcal {VH}^{k}_{(a,b)}(\nu,\lambda)$,
defined:
\begin{itemize}
\item $\{x(i),y(i)\}$ is free if every
$x(i)$ and $y(i)$ in $T$ share a cell
\item $\{x(i)\}$ is free if
every $x(i)\in T$ occurs alone and at the top of its column
\item $\{y(i)\}$ is free if every $y(i)\in T$ occurs alone 
and is not right-adj to $x$ or $y$.
\end{itemize}

Let $a=\binom{s+j-2}{j}$,
$b=\binom{s+j-3}{j}$, and  $c=\binom{s+j-3}{j-1}$.
For $T\in \mathcal {VH}^k_{s,r-s-j}(\nu,\lambda)$, let
$\mathcal S_{T}$ be the multiset containing $|a|$ copies of 
$(\text{sign(a)},T)$ if the lowest free entry is not $\{x\}$.  
Otherwise, let $\mathcal S_{T}$ be the ordered multiset with
$|b|$ copies of $(\text{sign(b)},T)$ followed by $|c|$ copies of 
$(\text{sign(c)},T)$.
Eq.~\eqref{epieri2} is then
\begin{equation}
\sum_{s=0}^{r} \sum_{j=0}^{r-s}
(-1)^{r-s-j}
\sum_{T\in \mathcal {VH}^k_{s,r-s-j}(\nu,\lambda)}
\;
\sum_{(\sigma,T)\in \mathcal S_{T}}
\sigma
=0 \,.
\end{equation}
If $\mathcal T^k_r{(\nu,\lambda)}$ is the
union of multisets $\mathcal S_{T}$, for all
$T\in\mathcal {VH}^k_{s,r-j-s}(\nu,\lambda)$ 
where $j,s\geq 0$ and $0\leq j+s \leq r$,
then the above expression reduces to 
\begin{equation}
\sum_{(\sigma,T)\in\mathcal T^k_r(\nu,\lambda)}
(-1)^{\weight_y(T)}
\times
\sigma
=0\,,
\end{equation}
where $\weight_y(T)$ is the number of residues occupied
by $y$'s.  Our result will follow from Property~\ref{minvo} which 
gives an involution $\mathfrak m$ 
on $\mathcal T^{k}_{r}(\nu,\lambda)$ where
$\mathfrak m(\sigma,T)=(\sigma,\hat T)$ with the
property that $\weight_y(T)=\weight_y(\hat T)\pm 1$.
\end{proof}

\begin{definition}
Define the map
$$
\mathfrak m: \mathcal T^k_{r}(\nu,\lambda)\to 
\mathcal T^k_{r}(\nu,\lambda)
$$
on $(\sigma,T)$ in position $p$ of $\mathcal S_T$ as follows:
let $s,j,m,i$ be integers where
$\weight(T)=(s,r-s-j)$  and the
lowest free entry of $T$ has residue $i$ and lies
in row $m$.
Then
$\mathfrak m(\sigma, T)=(\sigma,\hat T)$
is in position $\hat p$ where
\begin{enumerate}
\item 
if row $m$ contains a free $\{y(i)\}$, then
$\hat p=p$ and 
$\hat T$ is obtained by 
replacing all $\{y(i)\}$ in $T$ by
$\{x(i)\}$ 
\item
if row $m$ contains a free $\{x(i),y(i)\}$,
then 
$\hat p=|\binom{s+j-2}{j+1}|+p$
and $\hat T$ is obtained by 
replacing all
$\{x(i),y(i)\}$ in $T$ by $\{x(i)\}$ 
\item 
if row $m$ contains a free $\{x(i)\}$, then
\begin{enumerate}
\item
if $p\leq |\binom{s+j-3}{j}|$ 
then $\hat p=p$ and $\hat T$ is obtained by
replacing all $\{x(i)\}$ in $T$ by $\{y(i)\}$ 
\item otherwise $\hat p=p-|\binom{s+j-3}{j}|$ 
and $\hat T$ is obtained by replacing all $\{x(i)\}$ 
in $T$ by $\{x(i),y(i)\}$ 
\end{enumerate}
\end{enumerate}
\end{definition}

Let us emphasize that a $vh$-filling 
$T\in \mathcal {VH}^k_{(a,b)}(\nu,\lambda)$
is constructed as follows:
take the transpose of the tableau obtained by putting
$x$'s in an affine s-v $a$-strip added to $\core(\lambda)'$.
To the resulting tableau, add an affine s-v $b$-strip 
filled with $y$'s.  

\begin{lemma}
\label{free}
Given a $vh$-filling $T$,
if there is an $\{x(i),y(i)\}\in T$ that is not free then
there must be an $x(i-1)$ or a $y(i-1)$ in $T$.
\end{lemma}
\begin{proof}
Consider $T\in \mathcal {VH}^k_{(a,b)}(\nu,\lambda)$ with
a cell $c(i)$ containing a non-free $\{x,y\}$.
Given $x$ lies in cells of $\core(\mu)/\rho$ where
$(\mu^{\omega_k},\rho')\in\mathcal H^k_{\lambda^{\omega_k},a}$,
we have that
$\beta=\core(\mu)$ is the shape obtained by deleting all
lonely $y$'s from $T$.
Thus $c$ is $\beta$-removable since
no $x$ lies above a $y$ and the $x$'s form a vertical strip.

If $c'(i)$ contains a $\{y\}$ then
$c'\not\in\beta$ is above a cell in $\beta$ 
since the $y$'s form a horizontal strip.
Therefore the cell left-adj to $c'$ (of residue $i-1$) 
contains a $y$ since $\beta$ cannot have an addable and removable
$i$-corner.
On the other hand, if $c'(i)$ contains an $\{x\}$, assume 
$x(i-1)\not\in T$.  Then $c'$ is at the top of
its column in $\beta$ and is thus $\beta$-removable.
Further, $c\in \beta/\rho$ where $\rho$ is the shape obtained 
by deleting from $T$ any cell containing a $y$. Thus, all non-blocked
$\beta$-removable $i$-corners are in $\beta/\rho$ implying
that $c'$ is blocked (by a $y(i-1)$).
\end{proof}

\begin{property}
$\mathfrak m$ is well-defined
\end{property}
\begin{proof}
Consider $(\sigma,T)\in\mathcal T_r^k(\nu,\lambda)$.
By definition of free, no row of $T$ contains more than one free 
entry since $x$'s form a vertical strip in $vh$-fillings.
It thus suffices to show that $T$ contains 
a free $\{x\}$, a free $\{y\}$, or a free $\{x,y\}$.  
Suppose no entries are free.
$T$ contains an arbitrary letter $x(i)$ or $y(i)$ and
thus an entry $\{x(i)\}$, $\{y(i)\}$ or $\{x(i),y(i)\}$.
If $\{x(i)\}$ or $\{y(i)\}$ is in $T$, then it is not free
implies there is an $\{x(i),y(i)\}$, an $x(i-1)$, or a $y(i-1)$ 
in $T$.   On the other hand, if there is an
$\{x(i),y(i)\}\in T$, then Lemma~\ref{free} 
implies there is an $x(i-1)$ or a $y(i-1)$ in $T$.  
Therefore, there is an $x(i-1)$ or a $y(i-1)$ in $T$.
From this, the same argument implies there must 
be an $x(i-2)$ or a $y(i-2)$ in $T$.
By iteration, $T$ 
contains the letters $z(i) , z(i-1), z(i-2),\dots,z(i+2),z(i+1)$,
where each $z(t)$ is $x(t)$ or $y(t)$.
This contradicts that $T$ has weight $(s,r-j-s)$ for $r-j\leq k$.
\end{proof}

\begin{lemma}
\label{freenotin}
Given $T\in \mathcal T^k_r(\nu,\lambda)$ 
\begin{enumerate}
\item if $\{x(i)\}$ is the lowest free entry in $T$
then  $y(i)\not\in T$
\item if $\{y(i)\}$ is the lowest free entry in $T$
then  $x(i)\not\in T$
\end{enumerate}
\end{lemma}
\begin{proof}
Given $T\in \mathcal {VH}^{k}_{(a,b)}(\nu,\lambda)$,
$x$ lies in cells of $\core(\mu)/\rho$ where
$(\mu^{\omega_k},\rho')\in\mathcal H^k_{\lambda^{\omega_k},a}$
and letter $y$ lies in cells of $\core(\nu)/\tau$ where
$(\nu,\tau)\in\mathcal H^k_{\mu,b}$.  

(1): 
Any free $\{x(i)\}$ is a removable corner of $\core(\mu)$
since it lies at the top of its column and the $x$'s form a 
vertical strip.  Therefore, by Property~\ref{distinctres}, there 
can be no $y(i)$ in the affine strip $\core(\nu)/\core(\mu)$.
Suppose $y(i)\in \core(\mu)$. Then there is a $y$ in all 
removable $i$-corners of $\core(\mu)$ that are not $\core(\nu)$-blocked.   
Since the free $\{x(i)\}$ is not blocked, it shares
a cell with $y(i)$, violating the definition of free.

(2):
Given the lowest free entry is a $\{y(i)\}$ in cell $c_y$ of row $m$,
suppose there is an $x(i)\in T$ in cell $c_x$ of row $m_x$.
Since $\{y(i)\}$ is free, $c_x$ contains a {\it lonely}
$\{x(i)\}$.
In $\core(\mu)$, cell $c_x$ lies above an $i+1$-extremal by 
verticality of $x$'s and the cell beneath $c_y$ is at the top 
of its column by horizontality of $y$'s.  Therefore, by 
Property~\ref{thecoreprop}, $m_x>m$.  
Let $\beta$ be the shape obtained by
deleting from $T$ all lonely $y$'s {\it and} all lonely $x(j)$'s
for any residue $j$ that lies higher than the highest $x(i)$ 
and is occupied by an $x$.
Proposition~\ref{iteratesvstrip} implies that $\beta$ is a core
and $c_x(i)$ lies at the top of its column.  However, the $\{y\}$ in 
cell $c_y$ is free and thus is right-adj to a cell in $\beta$ that
lies above an $i$-extremal. We reach a contradiction by 
Property~\ref{thecoreprop}.  
\end{proof}

\begin{property}
\label{minvo}
The map $\mathfrak m$ is an involution on 
$\mathcal T_{r}^{k}(\nu,\lambda)$ and
for $\mathfrak m(\sigma,T)=(\sigma,\hat T)$,
$\weight_y(T)=\weight_y(\hat T)\pm 1$ 
\end{property}

\begin{proof}
Given $(\sigma,T)\in \mathcal T_{r}^{k}(\nu,\lambda)$,
define $p,r,s$ so that $(\sigma,T)$ is in 
position $p$ of $\mathcal S_T$ and $\weight(T)=(s,r-j-s)$.
We will show that 
$\mathfrak m(\sigma,T)=(\sigma,T_1)\in\mathcal T^k_r(\nu,\lambda)$,
$\weight_y(T_1)=r-j-s\pm 1$, and $\mathfrak m^2=1$.
Let $m$ denote the lowest row with a free entry in $T$
and set $a=\binom{s+j-2}{j}, b=\binom{s+j-3}{j}$, and
$c=\binom{s+j-3}{j-1}$.  

\smallskip

Consider the case when row $m$ of $T$  has a free $\{y(i)\}$.
Then $p\leq |\mathcal S_T|=|a|$  
and $\sigma=\text{sign(a)}$.  In this case, $T_1$ is obtained 
by replacing each $\{y(i)\}$ in $T$ by $\{x(i)\}$.  Any $y(i)\in T$ 
is lonely by definition of free and $x(i)\not\in T$ by 
Lemma~\ref{freenotin}.  Therefore, any $x(i)$ in $T_1$ is lonely 
and lies at the top of its column by horizontality of $y$'s. 
From this, the lowest free entry in $T_1$ is  the $\{x(i)\}$ in 
row $m$ and the weight of $T_1$ is $(s+1,r-j-s-1)$.
Thus $\mathcal S_{T_1}$ has $(\text{sign($a$)},T_1)$
in its first $|a|$ positions by definition
of $\mathcal S_{T_1}$.  In particular, there is a $(\sigma,T_1)$ 
in position $\hat p=p\leq |a|$ of $\mathcal S_{T_1}$.
Moreover, since the lowest free in $T_1$ is an $\{x(i)\}$
and $y(i)\not\in T_1$, $\mathfrak m$ acts on $(\sigma,T_1)$
by replacing each $\{x(i)\}$ by $\{y(i)\}$ and the
$(\sigma,T)$ in position $p$ is recovered.

\smallskip

In the case that row  $m$  of $T$ contains a free $\{x(i),y(i)\}$, 
we again have $\sigma=\text{sign(a)}$ and $p\leq |a|$.
$T_1$ is obtained by replacing each $\{x(i),y(i)\}$ in $T$ by $\{x(i)\}$.
The definition of free implies there are no lonely
$x(i)$ or $y(i)$ in $T$ and therefore
$T_1$ has weight $(s,r-j-s-1)$.
Further, the lowest free entry in $T_1$
is an $\{x(i)\}$ in row $m$ 
since each $\{x(i),y(i)\}$ in $T$ lies at the top of its
column by horizontality of $y$'s and is sent to a
lonely $x(i)$ in $T_1$.  This given, there are
$|\binom{s+j-2}{j+1}|+|a|$
elements of $\mathcal S_{T_1}$ of which the last
$|a|$ entries are $(\text{sign($a$)},T_1)$.
Thus, $p\leq |a|$ implies $(\text{sign($a$)},T_1)$ 
is in position $\hat p=p+|\binom{s+j-2}{j+1}|$.
Further, $\mathfrak m$ acts by replacing each $\{x(i)\}$ 
in $T_1$ by $\{x(i),y(i)\}$ and we have $\mathfrak m^2=id$.

\smallskip

The last case is when there is a free $\{x(i)\}$ in row $m$.
There are two scenarios depending on $p$.
When $p\leq |b|$, $\sigma=\text{sign(b)}$ and
$T_1$ is obtained by replacing $\{x(i)\}$ with $\{y(i)\}$ in $T$.
Since $y(i)\not\in T$ by Lemma~\ref{freenotin},
the weight of $T_1$ is $(s-1,r-j-s+1)$.
Further, the lowest free in $T_1$ is $\{y(i)\}$ in row $m$
since there is at most one $x$ in each row of $T$ implies
that no $x$ or $y$ is left-adj to $\{y(i)\}$ in $T_1$.
Thus, the $|b|$ entries of $\mathcal S_{T_1}$ are
$(\text{sign($b$)},T_1)$.  Therefore, there is a
$(\sigma,T_1)$ in position $\hat p =p$ of $\mathcal S_{T_1}$.
When $\mathfrak m$ acts on $(\sigma,T_1)$,
each $\{y(i)\}$ in $T_1$ is replaced by $\{x(i)\}$
and we recover $(\sigma,T)$ in position $p$.

Otherwise, $|b|+1\leq p\leq |b|+|c|$ and $\sigma=\text{sign(c)}$.
$T_1$ is obtained by replacing $\{x(i)\}$
by $\{x(i),y(i)\}$.  Since $y(i)\not\in T$ by Lemma~\ref{freenotin},
the weight of $T_1$ is $(s,r-j-s+1)$ and every $x(i)$ lies with $y(i)$ 
and vice versa.  Thus the lowest free entry is an $\{x(i),y(i)\}$
in row $m$ implying that $\mathcal S_{T_1}$ is $|c|$ copies of
$(\text{sign($c$)},T_1)$.  Therefore, 
there is a $(\sigma,T_1)$ in
position $p-|b|\leq |c|$ of $\mathcal S_{T_1}$.
When $\mathfrak m$ acts on $(\sigma,T_1)$,
each $\{x(i),y(i)\}$ in $T_1$ is replaced by $\{x(i)\}$
and we recover $(\sigma,T)$ in position $p$.
\end{proof}

\section{Conjugating affine Grothendieck polynomials}

An important property in the theory of Schur functions 
and $k$-Schur functions involves the algebra automorphism
defined on $\Lambda$ by $\omega: h_\ell\to e_\ell$.
Not only does $\omega$ send $s_\lambda$ to the
single Schur function  $s_{\lambda'}$, 
but it was proven in \cite{[LMproofs]} that
\begin{equation}
\omega\left(s_\lambda^{(k)}\right) = s^{(k)}_{\lambda^{\omega_k}}
\,,
\end{equation}
where $\lambda^{\omega_k} = \kbnd(\core(\lambda)')$.

\smallskip

In our study, we consider the algebra endomorphism defined on $\Lambda$
by
\begin{equation}
\Omega h_\ell =\sum_{j=1}^{\ell}\binom{\ell-1}{j-1}  e_{j}
\end{equation}
to be an inhomogeneous analog of $\omega$.
Note: the transformation 
$e_\ell \to \sum_{j=1}^{\ell}\binom{\ell-1}{j-1}  e_{j}$
has been studied \cite{BGI} and is needed to 
relate the cohomology ring to the Grothendieck ring.
In fact, the polynomials
$\sum_{j=1}^{\ell}\binom{\ell-1}{j-1}  e_{j}$
are connected to the study of classes of a Schubert 
subvariety of the Grassmannian in these rings \cite{LasSMF}.

\begin{remark}
A manipulatorial proof that $\Omega$ is an involution on 
$\Lambda^{(k)}$, supplied by Adriano Garsia, shows that
$$
\sum_{\ell\geq 1}\left(\frac{u}{u-1}\right)^\ell\Omega h_\ell = 
\sum_{\ell\geq 1}\left(\frac{u}{u-1}\right)^\ell
\sum_{j=1}^{\ell}\binom{\ell-1}{j-1}e_{j}
$$
$$
=
\sum_{j\geq 1}e_{j}
\sum_{\ell\geq j}
\binom{\ell-1}{j-1} 
\left(\frac{u}{u-1}\right)^{\ell-j}
\left(\frac{u}{u-1}\right)^j
=
\sum_{j\geq 1}e_{j}
\frac{\left(\frac{u}{u-1}\right)^j}
{1-\left(\frac{u}{u-1}\right)^j}
=\sum_{j\geq 1} (-1)^j e_j u^j
$$
implies by Newton's formula \eqref{newton} that
$$
\left(
\sum_{\ell\geq 0}u^\ell h_\ell\right)
\left(
\sum_{\ell\geq 0}\left(\frac{u}{u-1}\right)^\ell\Omega h_\ell \right)
= 1\,.
$$
The result follows by 
substituting $u=\frac{u}{u-1}$ into this expression and
applying $\Omega$:
$$
\left(
\sum_{\ell\geq 0}\left(\frac{u}{u-1}\right)^\ell\Omega h_\ell \right)
\left(
\sum_{\ell\geq 0}u^\ell\Omega^2 h_\ell\right)
=
1
\,.
$$

\end{remark}

\begin{remark}
By Jacobi-Trudi we have
$$\Omega(s_\lambda)=s_{\lambda'}+\text{lower degree terms}
\,.
$$
\end{remark}

\smallskip

This involution acts beautifully on the $k$-$K$-Schur functions
just as $\omega$ acts on a Schur function (or more generally, a 
$k$-Schur function).
\begin{theorem}
\label{coolinvo}
For any $k$-bounded partition $\lambda$,
\begin{equation}
\Omega g_\lambda^{(k)} = g_{\lambda^{\omega_k}}^{(k)} \, .
\end{equation}
\end{theorem}

\begin{proof}
Let $F_\lambda= \Omega\left(g^{(k)}_{\lambda^{\omega_k}}\right)$.  
Since $h_\ell \, \Omega\left(g^{(k)}_{\lambda}\right)
=\Omega\left(
\sum_{j=1}^{\ell}
\binom{\ell-1}{j-1} e_{j} g^{(k)}_{\lambda}\right)$, 
we can apply the column Pieri rule (Theorem~\ref{ekpieri}) to obtain
\begin{equation}
h_\ell\,F_\lambda=
\Omega\left(
g_{(1^\ell)}
g^{(k)}_{\lambda^{\omega_k}}\right)
=
\sum_{(\mu,\rho)\in \mathcal E^k_{\lambda^{\omega_k},\ell}}
\!  \!  \!  \!
 (-1)^{|\mu|-|\lambda|-\ell}\,
\Omega
g^{(k)}_{\mu}
=
\sum_{(\mu,\rho)\in \mathcal E^k_{\lambda^{\omega_k},\ell}}
\!  \!  \!  \!
 (-1)^{|\mu|-|\lambda|-\ell}\,
F_{\mu^{\omega_k}} 
\end{equation}
\begin{equation}
=
\sum_{(\mu^{\omega_k},\rho)\in \mathcal E^k_{\lambda^{\omega_k},\ell}}
 (-1)^{|\mu|-|\lambda|-\ell}\,
F_{\mu} 
=
\sum_{(\mu,\rho')\in \mathcal H^k_{\lambda,\ell}}
(-1)^{|\mu|-|\lambda|-\ell}\,
F_{\mu} 
\,.
\end{equation}
By Theorem~\ref{setvalued2},
the iteration of this expression from $F_0=\Omega g_0=1$ matches
the iteration of the row Pieri rule \eqref{kpierieq} from $g_0=1$.  
Thus, $F_\mu$ satisfies 
\begin{equation}
h_\lambda = \sum_{\mu\in\mathcal P^k\atop |\mu|\leq|\lambda|}
\mathcal K_{\mu\lambda}^{(k)}\, F_\mu
\end{equation}
implying that $F_\mu=g_\mu^{(k)}$ by Definition~\ref{kgrotdef}
of the $k$-$K$-Schur functions.
\end{proof}

The result can also be translated into the language of 
affine permutations since
$\core(\lambda)'=\core(\lambda^{\omega_k})$.

\begin{corollary}
For any $w\in \tilde S_{k+1}^0$,
$$
\Omega g_w^{(k)} = g_{w'}^{(k)} 
$$
where $w'$ is obtained by replacing
$s_i$ in $w$ with $s_{k+1-i}$.
\end{corollary}

\section{Computability}
The notion of affine s-v strips and Theorem~\ref{setvalued2}
give an efficient recursive algorithm to compute
$k$-$K$-Schur functions and affine Grothendieck polynomials.
This enabled us to check all conjectures extensively.

\bibliographystyle{halpha}

\bibliography{ktheoryjune}

\end{document}